\documentclass[final]{siamltex1213}

\usepackage{overpic}
\usepackage{amsmath,amstext,amsbsy,amssymb,mathdots}
\usepackage{tikz}
\usetikzlibrary{positioning}
\usetikzlibrary{shapes,arrows}
\usepackage{array,multirow}

\definecolor{DarkGreen}{rgb}{0,.55,0}

\definecolor{Yellow}{rgb}{.55,.55,0}

\definecolor{Red}{rgb}{.5,0,0}

\title{Fast polynomial transforms based on Toeplitz and Hankel matrices}
\author{Alex Townsend\thanks{Department of Mathematics, Massachusetts Institute of Technology, 77 Massachusetts Avenue
Cambridge, MA 02139-4307. (\texttt{ajt@mit.edu}) This work is supported by National Science Foundation grant No.~1522577.} 
\and Marcus Webb\thanks{Department of Applied Mathematics and Theoretical Physics, University of Cambridge, Wilberforce 
Road, Cambridge CB3 0WA, UK (\texttt{m.d.webb@maths.cam.ac.uk}) This work is supported by the UK Engineering and Physical Sciences Research Council (EPSRC) grant EP/H023348/1 for the University of Cambridge Centre for Doctoral Training, the Cambridge Centre for Analysis.} \and Sheehan Olver\thanks{School of Mathematics and Statistics, The University of Sydney, NSW 2006, Australia (\texttt{sheehan.olver@sydney.edu.au})}}
\date{\today}
\begin{document}
\maketitle

\begin{abstract} 
Many standard conversion matrices between coefficients in classical orthogonal polynomial expansions can be decomposed using
diagonally-scaled Hadamard products involving Toeplitz and Hankel matrices.  This allows us to derive
$\smash{\mathcal{O}(N(\log N)^2)}$ algorithms, based on the fast Fourier transform, for converting coefficients of a degree $N$ polynomial in one 
polynomial basis to coefficients in another. Numerical results show that this approach is competitive with state-of-the-art techniques, requires 
no precomputational cost, can be implemented in a handful of lines of code, and is easily adapted to extended precision arithmetic. 
\end{abstract}



\section{Introduction}
Expansions of polynomials as finite series in orthogonal polynomial bases have applications 
throughout scientific computing, engineering, and physics~\cite{Askey_75_01,Don_94_01,Shen_94_01}. The most popular expansions  are in 
the Chebyshev and Legendre basis,
\begin{equation}
p_N(x) = \sum_{k=0}^N c_k^{cheb} T_k(x) = \sum_{k=0}^N c_k^{leg} P_k(x), \qquad x\in[-1,1],
\label{eq:pN}
\end{equation} 
where $p_N$ is a degree $N$ polynomial and $T_k$ and $P_k$ are the degree $k$ Chebyshev and Legendre 
polynomials, respectively. Chebyshev expansions are used because of their near-optimal 
approximation properties and associated fast transforms~\cite{Gentleman_72_01,Mason_02_01,Trefethen_13_01} and Legendre expansions for their $L^2$ 
orthogonality~\cite[Table~18.3.1]{NISTHandbook} as well as other recurrence relations that they satisfy~\cite{Hale_14_02}.
A useful working paradigm is to represent a smooth function on a finite interval using the 
Chebyshev basis and to convert to a different polynomial basis, such as Legendre, whenever 
it is algorithmically convenient to do so~\cite{Hale_14_02}.  It is therefore important to have fast transforms for converting a 
vector of coefficients in one polynomial basis to another.  

Given two sequences of orthogonal polynomials $\{\phi_0,\phi_1,\ldots\}$ and $\{\psi_0,\psi_1,\ldots\}$, 
there is an $(N+1)\times (N+1)$ upper-triangular {\em conversion matrix} (sometimes called the connection coefficients matrix), $A$, such that  
\[
\underline{\beta} = A\underline{\alpha},\qquad p_N(x) = \sum_{k=0}^N \alpha_k \phi_k(x)=\sum_{k=0}^N \beta_k \psi_k(x),
\]
where $\underline{\alpha} = (\alpha_0,\ldots,\alpha_N)^T$ and $\underline{\beta} = (\beta_0,\ldots,\beta_N)^T$. 
In this paper we describe how to compute the matrix-vector product $A\underline{\alpha}$
in $\smash{\mathcal{O}(N(\log N)^2)}$ operations when $\{\phi_0,\phi_1,\ldots\}$ and $\{\psi_0,\psi_1,\ldots\}$ are sets of 
standard orthogonal polynomials such as Chebyshev, Legendre (see Sections~\ref{sec:HankelPart} and~\ref{sec:cheb2leg}), ultraspherical (see Section~\ref{sec:UltraphericalConversions}), Jacobi (see Section~\ref{sec:JacobiConversions}), and Laguerre (see Section~\ref{sec:AdditionalConversions}). 


There are many existing fast algorithms for computing $A\underline{\alpha}$ that exploit a variety of structures, 
including: (1) A hierarchical off-diagonal low rank structure~\cite{Alpert_91_01}, (2) An eigenvalue 
decomposition involving a semiseparable matrix~\cite{Keiner_09_01,Keiner_11_01,Potts_98_01}, and (3) Trigonometric-like behavior of orthogonal polynomials via 
asymptotic expansions~\cite{Hale_14_01,Mori_99_01,Orszag_86_01,slevinsky2016use}.   In this paper, we exploit a new observation that all standard conversion matrices can be decomposed using diagonally-scaled Hadamard products involving Toeplitz and Hankel matrices.\footnote{A Toeplitz matrix, $T$, is a matrix that is constant along each diagonal, i.e., $T_{jk} = t_{j-k}$.  A Hankel matrix, $H$, is a matrix that is constant along each anti-diagonal, i.e., $H_{jk} = h_{j+k}$.} This structure allows us to derive fast 
FFT-based algorithms for computing $A\underline{\alpha}$ that can be implemented in any programming language in 
just a handful of lines of code.  We believe this algorithm is conceptually simpler than previous approaches, 
while being competitive in terms of computational speed (see Sections~\ref{sec:NumericalResults} and~\ref{sec:OtherConversions}). 

It is easy to see the structure that we exploit by considering an example. Consider the $(N+1)\times (N+1)$ 
Legendre-to-Chebyshev conversion matrix, which converts a vector of Legendre coefficients to Chebyshev coefficients 
for a degree $N$ polynomial. It is given by~\cite[eqn.~(2.18)]{Alpert_91_01}
\begin{equation}
\underline{c}^{cheb}  = M  \underline{c}^{leg}, \qquad 
   M_{jk} = \begin{cases} 
   \frac{1}{\pi}\Lambda\left(\frac{k}{2}\right)^2,& 0=j\leq k\leq N, \text{ } j \text{ even},\\ 
   \frac{2}{\pi}\Lambda\left(\frac{k-j}{2}\right)\Lambda\left(\frac{k+j}{2}\right),& 0<j\leq k\leq N, \text{ } k-j \text{ even},\\     
   0, & \text{otherwise},                
  \end{cases}
\label{eq:M}
\end{equation}
where $\Lambda(z) = \Gamma(z+1/2)/\Gamma(z+1)$, $\Gamma(z)$ is the gamma function, and 
$\underline{c}^{cheb} = (c_0^{cheb},\ldots,c_N^{cheb})^T$ and  $\underline{c}^{leg} = (c_0^{leg},\ldots,c_N^{leg})^T$
are the Chebyshev and Legendre coefficients of $p_N(x)$ in~\eqref{eq:pN}, respectively.  The explicit formula in~\eqref{eq:M} reveals that after a trivial diagonal scaling, the matrix $M$ can be written 
as a Hadamard product between an upper-triangular Toeplitz matrix, from the $\Lambda(\tfrac{k-j}{2})$ term, and a Hankel matrix, from the 
term $\Lambda(\tfrac{j+i}{2})$.  Thus, for the matrix $M$ in~\eqref{eq:M} we can write
\begin{equation}
M = D_1(T\circ H)D_2,
\label{eq:HadamardDecomposition}
\end{equation} 
where $D_1$ and $D_2$ are diagonal matrices, $T$ is an upper-triangular Toeplitz matrix, $H$ is a Hankel matrix, and
`$\circ$' is the Hadamard matrix product, i.e., entrywise multiplication between two matrices (see~\eqref{eq:Decomposition} for 
explicit formulas for $D_1$, $D_2$, $T$ and $H$). We find that the structure in~\eqref{eq:HadamardDecomposition} holds for 
many of the standard conversion matrices (see Section~\ref{sec:OtherConversions}). 

At first it is not obvious why the decomposition in~\eqref{eq:HadamardDecomposition} is useful for deriving a 
fast matrix-vector product because for general Toeplitz and Hankel matrices we are not aware of a fast algorithm 
for computing $(T\circ H)\underline{v}$.  However, for conversion matrices the Hankel matrix in~\eqref{eq:HadamardDecomposition} is 
often real and positive semidefinite and hence, is severely ill-conditioned~\cite{Beckermann_00_01}.
Theorem~\ref{thm:LowNumericalRank} shows that the matrix $H$ in~\eqref{eq:HadamardDecomposition} can be approximated, up to an accuracy 
of $\epsilon>0$, by a rank $K = \mathcal{O}(\log N\log(1/\epsilon))$ matrix.  In practice, we take $\epsilon$ as a small multiple of machine epsilon.

For an integer $K$, we construct a rank $K$ approximation of $H$
in $\mathcal{O}(K^2N)$ operations using the pivoted Cholesky algorithm (see Section~\ref{sec:PivotedCholesky}) to obtain 
the approximation,
\begin{equation}
H \approx \sum_{r=1}^K a_r \, \underline{\ell}_r\, \underline{\ell}_r^T.
\label{eq:LowRankHankel}
\end{equation} 
Since $K$ needs to be no larger than $\mathcal{O}(\log N\log(1/\epsilon))$ we can compute~\eqref{eq:LowRankHankel} in a total of 
$\mathcal{O}(N(\log N)^2\log(1/\epsilon)^2)$ operations. 
Moreover, using $(A\circ \underline{\ell}\,\underline{\ell}^T)\underline{v} = (D_{\underline{\ell}}AD_{\underline{\ell}})\underline{v}$, where $D_{\underline{\ell}} = \mathrm{diag}(\underline{\ell})$, we can write
\begin{equation}
\left(T\circ H\right) \underline{v} \approx \sum_{r=1}^K a_r \left( D_{\underline{\ell}_r}T D_{\underline{\ell}_r} \right)\underline{v},
\label{eq:Cdecomposition}
\end{equation} 
where each term is a diagonally-scaled Toeplitz matrix whose matrix-vector product can be computed in $\mathcal{O}(N\log N)$ operations via the fast 
Fourier transform (FFT)~\cite[Sec.~3.4]{Ng_04_01}. There are $K$ terms in~\eqref{eq:Cdecomposition}, so the matrix-vector product 
$\left(T\circ H\right) \underline{v}$ can be computed, up to an error of $\mathcal{O}(\epsilon)$, in $\mathcal{O}(KN\log N+K^2N)$ operations. For the 
majority of this paper we write algorithmic complexities without the explicit dependency on $\epsilon$.  

\begin{figure} 
\centering
\fbox{\parbox{.95\textwidth}{
 {\bf Outline for fast polynomial basis conversion: Computing $\underline{\beta} = A\underline{\alpha}$.}\\[5pt]
 \begin{tabular}{llc} 
  & & {\bf Cost} \\ 
 1. & Decompose $A$ into $A = D_1(T\circ H)D_2$ (see Section~\ref{sec:OtherConversions}) & - \\[1pt]
 2. & Calculate $H \approx \sum_{r = 1}^K a_r\, \underline{\ell}_r\,\underline{\ell}_r^T$ (see Section~\ref{sec:PivotedCholesky}) & $\mathcal{O}(N(\log N)^2)$\\[1pt]
 3. & Compute $\underline{v} = D_2 \underline{\alpha}$ & $\mathcal{O}(N)$\\[1pt]
 4. & Compute $\underline{w} = (T\circ H)\underline{v}$ using~\eqref{eq:Cdecomposition} and the FFT~\cite{Ng_04_01} & $\mathcal{O}(N(\log N)^2)$\\[1pt]
 5. & Compute $\underline{\beta} = D_1\underline{w}$ & $\mathcal{O}(N)$\\\cline{3-3}
 \rule{0pt}{1.3\normalbaselineskip} & & $\mathcal{O}(N(\log N)^2)$
 \end{tabular} 
}}
\caption{A summary of the $\mathcal{O}(N(\log N)^2)$ algorithm for converting between many standard orthogonal polynomial bases, including 
Chebyshev, Legendre, ultraspherical, Jacobi, and Laguerre polynomial bases. 
For the algorithmic complexity of the fourth step see Section~\ref{sec:HankelLowRank}.}
\label{fig:Summary}
\end{figure} 

The approach for fast polynomial basis conversion, as outlined above, requires no hierarchical data structures, 
no precomputional setup cost, and no matrix partitioning.  The fundamental step is a fast Toeplitz matrix-vector 
product and the cost of the matrix-vector product is precisely $\mathcal{O}(K)$ FFTs of size $2N-1$~\cite[Sec.~3.4]{Ng_04_01}. The fact that our 
algorithm relies on the FFT means that the implementation is automatically tuned to personal computer 
architectures, thanks to FFTW~\cite{Frigo_05_01}.   Our algorithm is now the default algorithm for polynomial basis conversion in 
the MATLAB software system called Chebfun~\cite{Chebfun} (see the commands {\tt leg2cheb}, {\tt cheb2leg}, {\tt ultra2ultra}, and {\tt jac2jac}) and the Julia package called ApproxFun.jl~\cite{ApproxFun}, via the package FastTransforms.jl~\cite{FastTransforms} (see commands {\tt leg2cheb} and {\tt cheb2leg}).\footnote{Note to the reviewer: The Chebfun commands are currently under code review and will be added soon. The FastTransforms.jl commands are currently available via GitHub in the development branch and will soon be 
available in a tagged version of the software system.}

The paper is structured as follows.  In Section~\ref{sec:FastMatVec} we describe a pivoted Cholesky algorithm for constructing 
low rank approximations of real, symmetric, and positive semidefinite matrices and use it to compute $(T\circ H)\underline{v}$ in $\mathcal{O}(N(\log N)^2)$ 
operations, where $T$ is an $(N+1)\times (N+1)$ Toeplitz matrix and $H$ is a real, symmetric, and positive semidefinite Hankel matrix. 
In Section~\ref{sec:HankelPart} we write $M = D_1(T\circ H)D_2$ as in~\eqref{eq:HadamardDecomposition} 
and show that the Hankel part, $H$, can be approximated, up to an error of $\epsilon>0$, by a 
rank $K = \mathcal{O}( \log N\log(1/\epsilon))$ matrix. In Section~\ref{sec:NumericalResults} we compare various algorithms 
for converting Legendre-to-Chebyshev basis conversion. In Section~\ref{sec:OtherConversions} we show that our algorithm 
allows for fast matrix-vector products involving many of the standard conversion matrices and we give numerical results 
throughout that section. 

\section{A fast matrix-vector product for certain Toeplitz-dot-Hankel matrices}\label{sec:FastMatVec}
We say that a matrix $A$ is a {\em Toeplitz-dot-Hankel} matrix if $A$ can be written as a Hadamard product of a Toeplitz and Hankel matrix, i.e., 
$A = T \circ H$, where $T$ is a Toeplitz matrix, $H$ is a Hankel matrix, and `$\circ$' denotes the Hadamard product. 
In this section, we suppose that $H$ is a real, symmetric, and positive semidefinite Hankel matrix and that it is approximated, up to an error of $\epsilon>0$, by a rank $K$ 
matrix.   Later, in Section~\ref{sec:HankelPart} we show that this holds for the 
Hankel part of the Legendre-to-Chebyshev conversion matrix in~\eqref{eq:M} when $K = \mathcal{O}( \log N\log(1/\epsilon) )$. 

Once we have constructed the low rank approximation in~\eqref{eq:LowRankHankel}, costing $\mathcal{O}(K^2N)$ operations, a fast matrix-vector product is immediate as~\eqref{eq:Cdecomposition} shows that $(T\circ H)\underline{v}$ can be computed as a sum of matrix-vector 
products involving a Toeplitz matrix. 
 
\subsection{Low rank approximation of a real positive semidefinite matrices}\label{sec:PivotedCholesky}
Suppose that $H$ is approximated up to an error of $\epsilon$ by a rank $K$ matrix. We would like to compute a rank $K$ 
approximation to $H$, i.e., $a_1,\ldots,a_K$ and $\underline{\ell}_1,\ldots,\underline{\ell}_K$ in~\eqref{eq:LowRankHankel}.  

In principle, one could construct a best rank $K$ approximation of $H$ by computing the singular value decomposition (SVD) of $H$ 
and taking the first $K$ left and right singular vectors as well as the first $K$ singular values. Naively, this costs a prohibitive $\mathcal{O}(N^3)$ operations, 
which can be reduced to $\mathcal{O}(N^2\log N)$ operations if the Hankel structure is exploited~\cite{Xu_08_01}.   
Instead, we describe an algorithm that costs only $\mathcal{O}( K^2 N )$ operations based on a pivoted Cholesky algorithm. It 
can be applied to any real, symmetric, and positive semidefinite matrix and does not exploit the Hankel structure of $H$. In practice, 
if the singular values of $H$ decay rapidly, then the rank $K$ approximation constructed by our pivoted Cholesky algorithm 
is observed to be near-best (see Figure~\ref{fig:CompareCholeskyAndSVD}). 
Our algorithm is very similar to the pivoted Cholesky algorithm described in~\cite{Harbrecht_12_01}, except we avoid square roots 
and have a different stopping criterion. One can trace back the origin of our algorithm to the rank revealing literature~\cite{Gu_04_01,Higham_90_01}.  

Set $H^{(0)} = H$ and assume that $H$ is a nonzero, real, symmetric, and positive semidefinite matrix. First, the 
maximum on the diagonal of $H$ is selected, say $H_{p_1p_1}$, which is also the absolute global maximum entry 
of $H$~\cite[Thm.~4.2.8]{Golub_13_01}. Then, one step of the Cholesky algorithm is performed with 
the entry $(p_1,p_1)$ as the pivot, i.e., 
\begin{equation}
H^{(1)} = H^{(0)} - H_{:,p_1}^{(0)}H_{p_1,:}^{(0)}/H_{p_1p_1}^{(0)},
\label{eq:FirstCholeskyStep}
\end{equation} 
where $H_{:,p_1}^{(0)}$ and $H_{p_1,:}^{(0)}$ denotes the $p_1$th column and $p_1$th row of $H^{(0)}$, respectively. 
The matrix $H^{(1)}$ has a zero $p_1$th column and $p_1$th row. 
The Cholesky step in~\eqref{eq:FirstCholeskyStep} is closely related
to the Schur complement of the $(p_1,p_1)$ entry in $H^{(0)}$ and the matrix $H^{(1)}$ is guaranteed to be real, symmetric, and positive 
semidefinite. Furthermore, the rank of $H^{(1)}$ is exactly one less than $H^{(0)}$~\cite[Cor.~19.2]{Matsaglia_74_01}. 

Next, the maximum on the diagonal of $H^{(1)}$ is found, say $H_{p_2p_2}^{(1)}$. If $H_{p_2p_2}^{(1)} = 0$ or is sufficiently small, then the 
process is terminated; otherwise, another Cholesky 
step is taken with the entry $(p_2,p_2)$ as the pivot, i.e.,
\[
H^{(2)} = H^{(1)} - H_{:,p_2}^{(1)}H_{p_2,:}^{(1)}/H_{p_2p_2}^{(1)}. 
\]
Again, the matrix $H^{(2)}$ is guaranteed to be real, symmetric, and positive semidefinite and has a rank that is 
exactly one less than $H^{(1)}$.  The pivoted Cholesky algorithm continues until the maximum on diagonal is 
either zero or sufficiently small. Since the rank of $H$ is at most $N$ and the rank decreases by precisely one 
after each Cholesky step, the algorithm terminates in at most $N$ steps.  For the algorithm to be 
computationally more efficient than the SVD one hopes to need just $K\ll N$ steps.   

Suppose that the pivoted Cholesky algorithm takes $K$ steps before terminating. Since $H^{(0)},\ldots, H^{(K-1)}$ are 
symmetric matrices, we can write the $r$th Cholesky step as follows: 
\[
H^{(r)} = H^{(r-1)} - a_r \, \underline{\ell}_r\,\underline{\ell}_r^T, \qquad \underline{\ell}_r = H^{(r-1)}_{:,p_{r}}, \quad a_r = (H^{(r-1)}_{p_{r}p_{r}})^{-1}. 
\]
Therefore, we use the $r$th pivoted Cholesky step to construct the $r$th term in~\eqref{eq:LowRankHankel}, where after $K$ steps 
the rank $K$ approximation to $H$ is constructed.  

The pivoted Cholesky algorithm described so far requires a total of $\mathcal{O}(K N^2)$ operations
because at each step an $(N+1)\times (N+1)$ matrix is updated.  Now we will describe how to construct the same rank $K$ 
approximant to $H$ in $\mathcal{O}(K^2N)$ operations, which is a significant saving when $K\ll N$.  

The main idea to reduce the computational cost is to note that it is not necessary to update the 
whole matrix at each Cholesky step.  For example, consider the Hankel part, $H$, of the matrix $M$ in~\eqref{eq:M} for $N = 300$. 
Figure~\ref{fig:PivotedCholeskyLocations} shows the pivot locations selected by the pivoted Cholesky algorithm, where  
a total of $27$ steps was required to construct a low rank approximant of $H$ that is accurate up to double precision. 
This means that only $27$ columns from $H$ are required to compute~\eqref{eq:LowRankHankel} (see black vertical lines on Figure~\ref{fig:PivotedCholeskyLocations}).  Therefore, we 
rewrite the pivoted Cholesky algorithm so that it only updates the diagonal entries (required to determine the pivot locations) and those
$27$ columns of $H$.  This allows for a significant computational saving when $K\ll N$.  

\begin{figure} 
\centering 
\begin{overpic}[width=.6\textwidth]{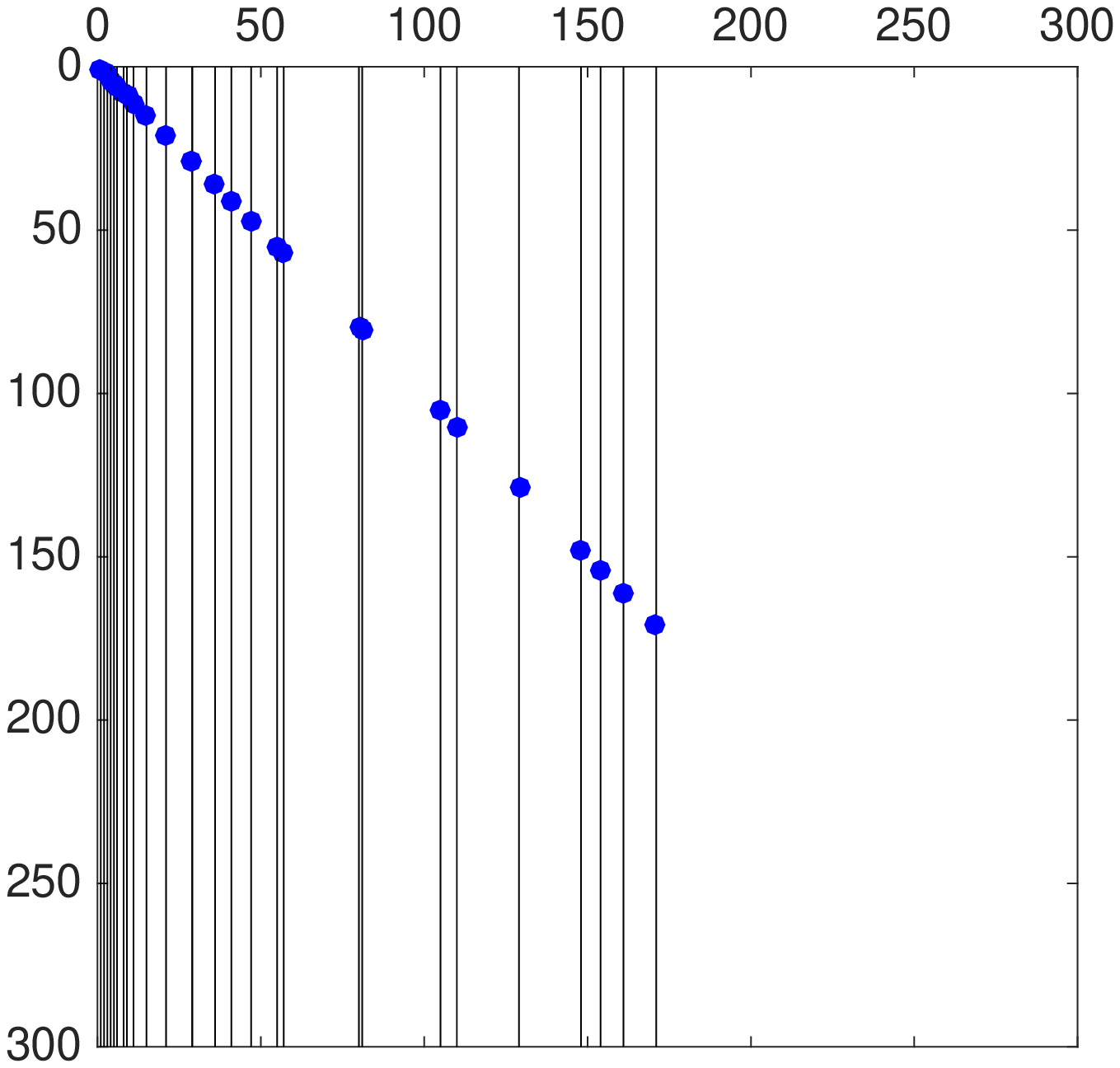} 
\end{overpic} 

\vspace{-.5cm}

\caption{The pivoted Cholesky algorithm applied to the Hankel part of the Legendre-to-Chebyshev conversion matrix in~\eqref{eq:M} when $N = 300$. The Cholesky algorithm takes $27$ pivots (blue dots) and evaluates the matrix at the $27$ corresponding columns (black lines). 
One can see that the majority of the entries in the Hankel part of $M$ are not required during the pivoted Cholesky algorithm, allowing 
for the computational cost to be reduced from $\mathcal{O}(N^2)$ to $\mathcal{O}(N)$ operations.} 
\label{fig:PivotedCholeskyLocations} 
\end{figure} 

Let $\underline{d} = \left(H_{00},\ldots,H_{NN}\right)^T$ be the diagonal entries of $H$.  In the first step, instead of~\eqref{eq:FirstCholeskyStep}, 
we only update the diagonal as follows:
\[
\underline{d} \quad \longleftarrow \quad \underline{d}  - (H_{:,p_1}^{(0)})^2/H_{p_1p_1}^{(0)},
\] 
where we have used the fact that $H^T = H$. The diagonal can then be used to determine the location of the second pivot. 
For the second step, we again only update the diagonal,
\[
\underline{d} \quad \longleftarrow \quad \underline{d}  - (H_{:,p_2}^{(1)})^2/H_{p_2p_2}^{(1)},
\]
where  $(H_{:,p_2}^{(1)})^2$ means that the vector $H_{:,p_2}^{(1)}$ is squared entry-by-entry. 
Since the pivot locations at each step are determined by the diagonal entries, one can select the pivoting entries 
by only updating the diagonal vector $\underline{d}$. At the $r$th Cholesky step the column $H_{:,p_r}^{(r-1)}$
is required, which is not directly available from the matrix $H$. We calculate this by first constructing 
$H_{:,p_r}$ and by applying each of the previous $r-1$ Cholesky steps to $H_{:,p_r}$ (see also~\cite{Harbrecht_12_01}).  
Figure~\ref{fig:CholekyDiagram} presents a summary of the algorithm. A simple operation 
count reveals that the algorithm costs $\mathcal{O}(K^2N)$ operations.    

\begin{figure} 
\centering
\fbox{\parbox{.95\textwidth}{
 {\bf Pivoted Cholesky algorithm for real, symmetric, and positive semidefinite matrices}\\[5pt]
 {\bf Input:} A real, symmetric, and positive semidefinite matrix, $H$, and an accuracy goal $\epsilon>0$. \\[2pt]
 {\bf Output:} Vectors $\underline{\ell}_1,\ldots,\underline{\ell}_K$ and $a_1,\ldots,a_K$ such that $H \approx \sum_{r=1}^K a_r\,\underline{\ell}_r\,\underline{\ell}_r^T$. \\[2pt]
  {\bf Set} $\underline{d} = (H_{11},\ldots,H_{NN})^T$, the diagonal of $H$ \\[2pt]
 {\bf for} $r = 1,\ldots, K$ \\[2pt]
  \-\hspace{1cm}       $ p_r = \arg\max_{1\leq j \leq N} d_j$ \\[2pt]
  \-\hspace{1cm}       {\bf Set} $\underline{\ell}_r = H_{:,p_r}$\\[2pt]
  \-\hspace{1cm}	   {\bf for} $j = 1,\ldots, r-1$   \\[2pt]      
  \-\hspace{2cm}          $\underline{\ell}_r \quad \longleftarrow \quad \underline{\ell}_r - \underline{\ell}_j (\underline{\ell}_r)_j /(\underline{\ell}_r)_{p_r}$\\[2pt]
  \-\hspace{1cm}     {\bf end}\\[2pt]
  \-\hspace{1cm}     {\bf Set} $a_{r} = (\underline{\ell}_r)_{p_r}$ \\[2pt] 
  \-\hspace{1cm}      $\underline{d} \quad \longleftarrow \quad \underline{d} - (\underline{\ell}_{r})^2/a_r$\\[2pt]
 {\bf end}
}}
\caption{Pseudocode for pivoted Cholesky algorithm, which costs $\mathcal{O}(K^2N)$ operations. 
This algorithm only requires the matrix $H$ to be real, symmetric, and positive semidefinite, not necessarily Hankel.
It can be shown that provided the singular values $\sigma_1,\ldots,\sigma_{N+1}$ of $H$ decay at a 
geometric rate that is faster than $\smash{4^{-k}}$, then the constructed rank $K$ approximant converges geometrically 
to $H$~\cite[Thm.~3.2]{Harbrecht_12_01}.}
\label{fig:CholekyDiagram}
\end{figure} 

One may be legitimately concerned that the low rank approximants constructed by the pivoting Cholesky algorithm in Figure~\ref{fig:CholekyDiagram}
are of poor quality, as they are not strictly best low rank approximants. More precisely, suppose that $H_K$ is the best rank $K$ approximant of $H$ in the matrix $2$-norm computed via the SVD 
and $\tilde{H}_K = \sum_{r=1}^K a_r\,\underline{\ell}_r\,\underline{\ell}_r^T$ is constructed via the pivoted Cholesky algorithm. 
    The best mathematical statement we know of is~\cite[Thm.~3.2]{Harbrecht_12_01}, which states that, provided that the singular values $\sigma_1,\ldots,\sigma_{N+1}$ of $H$ decay at a 
geometric rate that is faster than $\smash{4^{-k}}$, then the constructed rank $K$ approximant converges geometrically 
to $H$.  However, the assumptions of their theorem are considered
to be much stronger than necessary (see~\cite[Chap.~4]{Townsend_15_01}).  Is it possible that $\|H - H_K\|_2 \ll \|H-\tilde{H}_K\|_2$ otherwise? When $H$ has moderately decaying singular values, we believe not. 
One representative numerical experiment is shown in Figure~\ref{fig:CompareCholeskyAndSVD} 
(left), where the low rank approximants constructed by the pivoted Cholesky algorithm are compared against those from the SVD 
for the Hankel part of the matrix $M$ in~\eqref{eq:M} when $N = 1000$.  
All other investigations have revealed similar results~\cite{Harbrecht_12_01,Townsend_15_01}. 
A precise theorem that adequately describes the power of the pivoted Cholesky algorithm for constructing low rank approximants is 
a remaining mathematical challenge and may require a much deeper understanding on the numerical stability of Gaussian elimination. 
In principle, we could have presented a variant on the pivoted Cholesky algorithm to ensure that it is a
so-called strong rank-revealing algorithm~\cite{Gu_04_01}. However, this makes the pivoted Cholesky algorithm have a 
cost of $\mathcal{O}(N^2)$ operations and such a modification seems unnecessary in practice here.

\begin{figure} 
\begin{minipage}{.49\textwidth} 
\begin{overpic}[width=\textwidth]{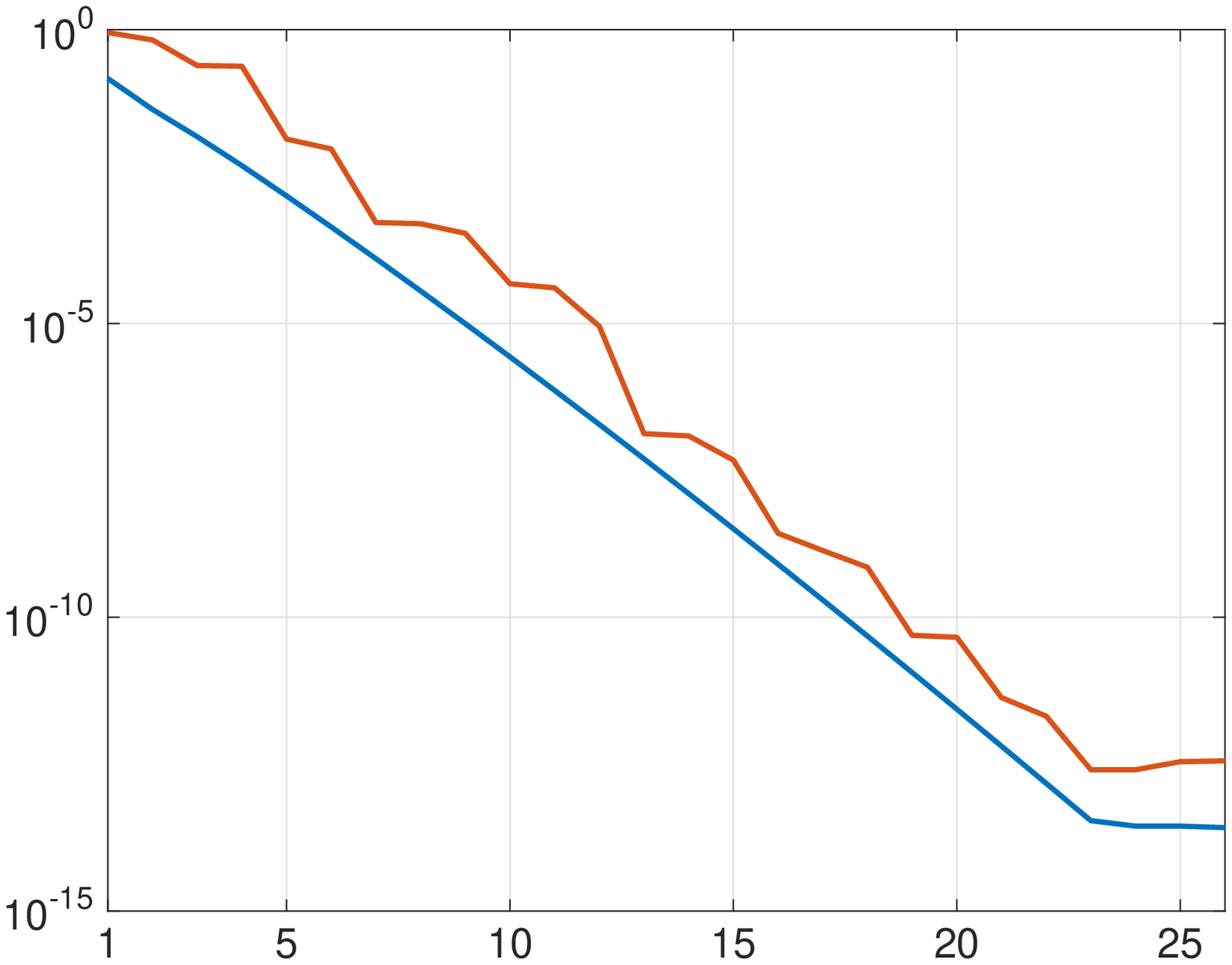}
\put(43,55) {\rotatebox{-40}{pivoted Cholesky}}
\put(45,35) {\rotatebox{-40}{SVD}}
\put(0,25) {\rotatebox{90}{$2$-norm error}} 
\put(47,0) {Rank} 
\put(10,73) {Cholesky's near-best approximants} 
\end{overpic} 
\end{minipage} 
\begin{minipage}{.49\textwidth} 
\begin{overpic}[width=\textwidth]{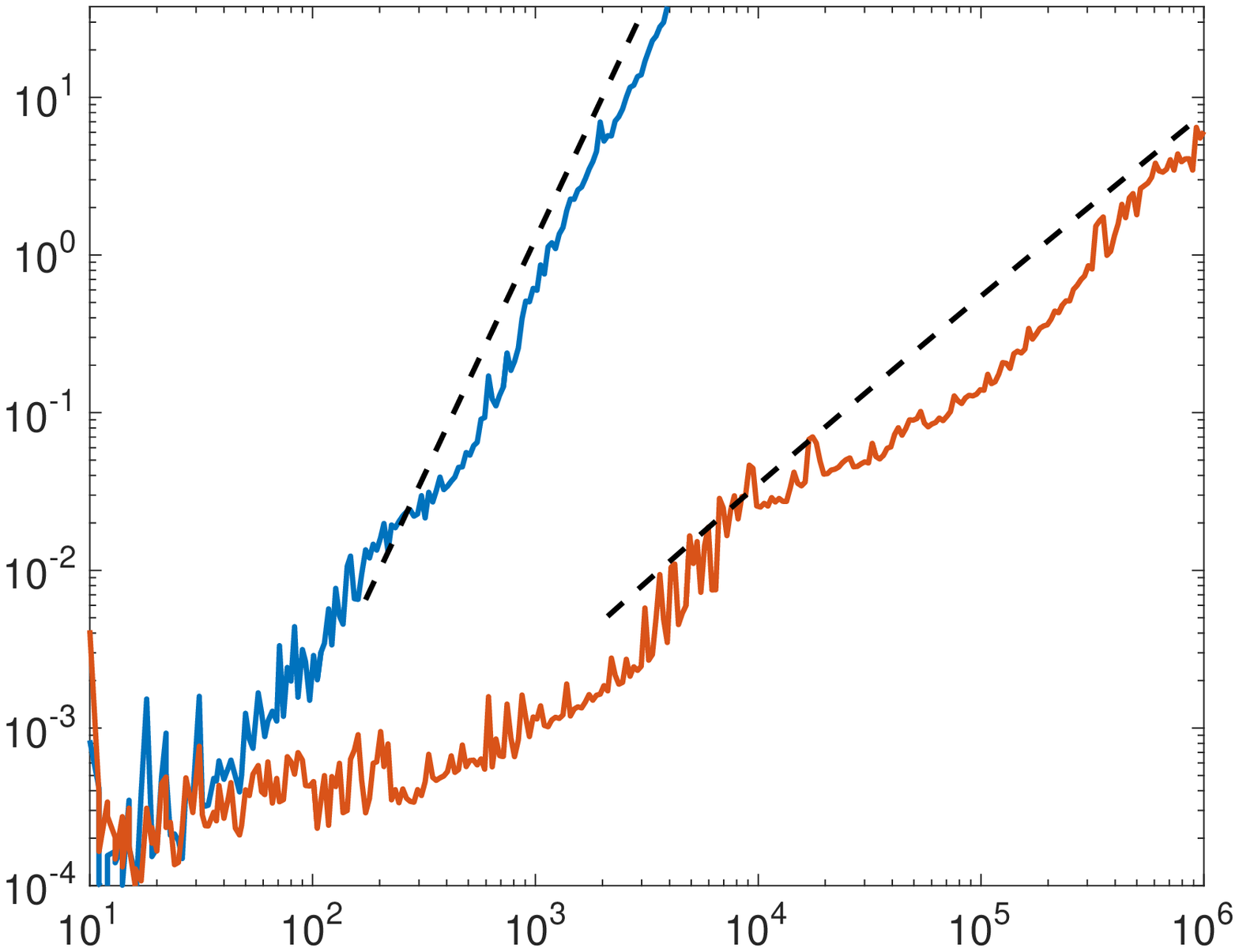}
\put(55,40) {\rotatebox{40}{$\mathcal{O}(N(\log N)^2)$}}
\put(32,45) {\rotatebox{63}{$\mathcal{O}(N^3)$}}
\put(57,23) {\rotatebox{40}{pivoted Cholesky}}
\put(45,45) {\rotatebox{63}{SVD}}
\put(20,73) {Cholesky's execution time}
\put(50,0) {$N$} 
\put(0,10) {\rotatebox{90}{Execution time (seconds)}} 
\end{overpic} 
\end{minipage} 
\caption{Left: The low rank approximants constructed by the pivoted Cholesky are often near-best. Here, the Hankel part of $M$ in~\eqref{eq:M} 
is considered when $N = 1000$ and the $2$-norm error between the SVD and pivoted Cholesky rank $K$ approximants is compared.  
Right: The pivoted Cholesky algorithm is computationally more efficient than the SVD at constructing low rank approximants. Here, the execution 
times are compared for the Hankel part of $M$ when $10\leq N\leq 10^6$. At each $N$ the algorithms construct a low rank approximant that 
is accurate to essentially double precision. Since $K = \mathcal{O}( \log N )$ for the Hankel part of $M$, the cost of the pivoted Cholesky algorithm is $\mathcal{O}( N(\log N)^2 )$ operations, while the cost of the SVD is $\mathcal{O}(N^3)$ operations.}
\label{fig:CompareCholeskyAndSVD}
\end{figure} 

Our pivoted Cholesky algorithm costs $\mathcal{O}(K^2N)$ operations. In Figure~\ref{fig:CompareCholeskyAndSVD} (right) we compare the 
execution time for computing the SVD and the pivoted Cholesky algorithm on the Hankel part of the matrix $M$ in~\eqref{eq:M} for 
obtaining an accuracy of essentially double precision. One can see that even when $N = 10^6$, the pivoted Cholesky algorithm can be 
employed to construct~\eqref{eq:LowRankHankel}, whereas the SVD is limited to $N\leq 5000$. In some applications the size of the matrix $M$ is 
fixed and one wishes to convert between two polynomial bases for many different polynomials of the same degree.  In such a situation, the pivoted Cholesky 
algorithm can construct a low rank approximation for the Hankel part just once and it can be reused for each matrix-vector product. 


%
%

\section{The Hankel part of the Legendre-to-Chebyshev matrix}\label{sec:HankelPart} 
Let $N$ be an integer. The $(N+1)\times (N+1)$ Legendre-to-Chebyshev matrix, denoted by $M$, in~\eqref{eq:M} 
can be written as a diagonal-scaled Toeplitz-dot-Hankel matrix. That is, $M = D_1(T\circ H)D_2$, where $D_1 = \tfrac{2}{\pi}{\rm diag}(\tfrac{1}{2},1,\ldots,1)$, $D_2=I_{N+1}$ the $(N+1)\times(N+1)$ identity matrix, and
\begin{equation}
\begin{aligned}
T_{jk} &= \begin{cases} \Lambda\!\left(\frac{k-j}{2}\right), & 0\leq j\leq k\leq N, k-j \text{ even},\\ 0, & \text{otherwise},\end{cases} \\
H_{jk} &=  \Lambda\!\left(\tfrac{j+k}{2}\right), \quad 0\leq j,k\leq N. 
\label{eq:Decomposition} 
\end{aligned}
\end{equation}
Here $\Lambda(z) = \Gamma(z+1/2)/\Gamma(z+1)$, where $\Gamma(z)$ is the gamma function.  In this section we show that 
$H$ is: (1) real, symmetric, and positive semidefinite (see Section~\ref{sec:HankelPosDef}) and (2) $H$ can be approximated, up to an accuracy of $\epsilon>0$, 
by a rank $K = \mathcal{O}(\log N\log(1/\epsilon))$ matrix (see Section~\ref{sec:HankelLowRank}).  

\subsection{The Hankel part of the Legendre-to-Chebyshev matrix is positive semidefinite}\label{sec:HankelPosDef} 
The Hankel matrix $H$ in~\eqref{eq:Decomposition} is immediately seen to be real and symmetric. 
To show that it is positive semidefinite, we recall that the Hamburger moment problem states that a real 
Hankel matrix is positive semidefinite if and only if it is associated to a nonnegative Borel measure 
supported on the real line. 
\begin{lemma}[Hamburger Moment Problem]\label{lem:hamburger}
 A real $(N+1) \times (N+1)$ Hankel matrix, $H$, is positive semidefinite if and only if there exists a 
 nonnegative Borel measure $\mu_H$ supported on the real line such that
 \begin{equation}
  H_{jk} = \int_{-\infty}^{\infty} x^{j+k} \mathrm{d}\mu_H(x), \qquad 0\leq j,k\leq N.
  \label{eq:moments}
 \end{equation}
 \end{lemma}
 \begin{proof}
  For a proof, see~\cite[Thm.~7.1]{Peller_12_01}. 
 \end{proof}

We show that the Hankel matrix in~\eqref{eq:Decomposition} is positive semidefinite 
by expressing its entries in the form of~\eqref{eq:moments}. 
\begin{theorem}
The Hankel matrix, $H$, in~\eqref{eq:Decomposition} is positive semidefinite. 
\label{thm:PosDefHankelMatrix} 
\end{theorem} 
\begin{proof}
It can be verified that (see~\cite[Sec.~7]{Tricomi_57_01} for integral representations of ratios of gamma functions)
\[
H_{jk} = \frac{2}{\sqrt{\pi}} \int_0^1 x^{j+k}(1-x^2)^{-1/2} dx, \qquad 0\leq j,k\leq N. 
\]
By setting $d\mu_H(x) = (2/\sqrt{\pi})\chi_{(0,1)}(1-x^2)^{-1/2}dx$ in Lemma~\ref{lem:hamburger}, where $\chi_{(0,1)}$ is the characteristic function for the interval $(0,1)$, we conclude that $H$ is a positive semidefinite matrix. 
\end{proof} 

Theorem~\ref{thm:PosDefHankelMatrix} shows that $H$ is positive semidefinite and therefore, the pivoted Cholesky 
algorithm described in Section~\ref{sec:PivotedCholesky} is applicable. We now show that $H$ can be 
well-approximated by a rank $K$ matrix where $K\ll N$. 

\subsection{The Hankel part of the Legendre-to-Chebyshev matrix is numerically of low rank}\label{sec:HankelLowRank} 
In Section~\ref{sec:HankelPosDef} we showed that $H$ in~\eqref{eq:Decomposition} is real, symmetric, and positive semidefinite. 
Such Hankel matrices are severely ill-conditioned, and
Beckermann has proved the remarkably tight bound of $\kappa_2(H) \geq 0.0037 \times (33.97)^{N}/N$~\cite{Beckermann_00_01}, where 
$\kappa_2(H)$ denotes the condition number of $H$ in the matrix $2$-norm. This shows that $H$ is in fact exponentially ill-conditioned, so one might expect 
that the singularly values $\sigma_1(H),\ldots,\sigma_{N+1}(H)$ of $H$ decay geometrically to zero. Indeed they do decay exponentially, and in this section we show that $H$ can be approximated to an accuracy of $\epsilon>0$ by a 
rank $K = \mathcal{O}( \log N \log(1/\epsilon) )$ matrix.   

First, we extract out the first row of $H$. That is, 
\begin{equation}
H = \begin{bmatrix} H_{00}\!\! & \!\! \cdots\!\! & \!\!  H_{0N} \\[5pt] &\hat{H}&  \end{bmatrix}, 
\label{eq:SplitFirstRow}
\end{equation} 
where $\hat{H}$ is an $N\times (N+1)$ submatrix of $H$. Then, we view $\hat{H}$ as being obtained by sampling the 
function $\Lambda((x+y)/2)$ at the tensor grid $\{0,\ldots,N\}\times \{1,\ldots,N\}$. A similar strategy 
is employed by Grasedyck to show that the singular values of the Cauchy function decay geometrically to zero~\cite{Grasedyck_01_01}. 
We build a function $h(x,y)$ defined on $[0,N]\times [1,N]$ that is of rank $\mathcal{O}( \log N\log(1/\epsilon) )$ 
such that 
\[
\sup_{(x,y) \in [0,N]\times [1,N]} \left| \Lambda\left(\tfrac{x+y}{2}\right) - h(x,y) \right| \leq \epsilon. 
\]
By sampling $h(x,y)$ at the grid $\{0,\ldots,N\}\times \{1,\ldots,N\}$ we construct a matrix of 
rank $\mathcal{O}( \log N\log(1/\epsilon) )$ that approximates $\hat{H}$ to an accuracy of $\epsilon$. 
We construct $h(x,y)$ by geometrically dividing the domain $[0,N]\times [1,N]$ at $y = 1$ into domains
$[0,N]\times [1,\beta^{k} N],  [0,N]\times [\beta^k N,\beta^{k-1} N], \ldots, [0,N]\times [\beta N,N]$, 
where $1/2<\beta<1$ is a parameter to be selected (see Figure~\ref{fig:GeometricSubdivision}). 

\begin{figure} 
\centering
\includegraphics[width=.6\textwidth]{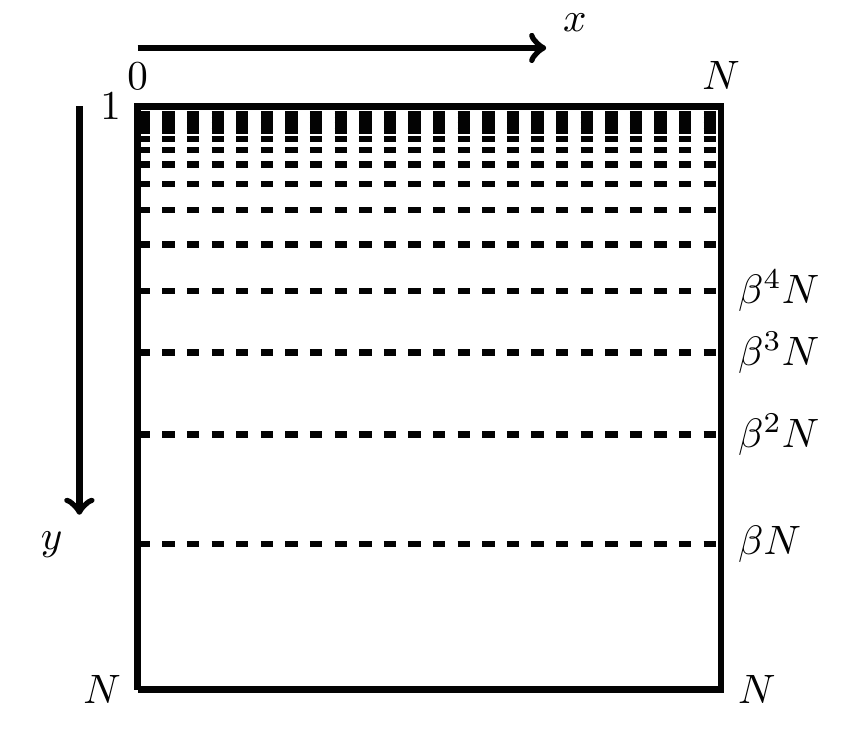}
\caption{The geometric subdivision of $[0,N]\times [1,N]$ that is used to show that the Hankel part, $H$, of $M$ in~\eqref{eq:M} can be approximated,
up to an error of $\epsilon$, by a rank $K = \mathcal{O}(\log N\log(1/\epsilon))$ matrix. First, we extract off the first row of $H$ and view the matrix $\hat{H}$ in~\eqref{eq:SplitFirstRow} as derived from sampling $\Lambda((x+y)/2)$ at equally-spaced points on $[0,N]\times [1,N]$.  Motivated by~\cite{Grasedyck_01_01}, we geometrically subdivide the domain $[0,N]\times [1,N]$ at $y = 1$ and bound the numerical rank of $\Lambda((x+y)/2)$ on each domain by using Taylor series expansions (see Lemma~\ref{lem:LowNumericalRankSections}).}
\label{fig:GeometricSubdivision} 
\end{figure} 

On each domain of the form $[0,N]\times [\beta^j N,\beta^{j-1} N]$ we can use Taylor expansions to relatively 
tightly bound the numerical rank of the function $\Lambda((x+y)/2)$. Later, we will sum these ranks together to 
obtain a bound on the numerical rank of $\Lambda((x+y)/2)$ on $[0,N]\times[1,N]$. 

Our first task is to control the error from a Taylor expansion and to do this we bound the $y$-derivative of $\Lambda((x+y)/2)$. 
\begin{lemma} 
Let $n\geq 0$ and $N>1$ be integers. Then, 
\[
\left| \frac{\partial^n}{\partial y^n}\Lambda\left(\tfrac{x+y}{2}\right)\right| \leq \frac{n!\, e}{y^n}, \qquad (x,y) \in [0,N]\times[1,N].
\]
\label{lem:BoundedDerivative} 
\end{lemma} 
\begin{proof}
Note that $\Lambda(z) = \Gamma(z+1/2)/\Gamma(z+1)$ is an analytic function in $\mathbb{C}$ 
except at the poles $z = -1/2, -3/2, \ldots$.  Therefore, 
for each fixed $y\in [1,N]$ the function $\Lambda((z+y)/2)$ is analytic except at  $z = -1/2-y, -3/2-y,\ldots$. 
Hence, $\Lambda((z+y)/2)$ is analytic in the disk in $\mathbb{C}$ of radius $y$ centered at $z = y$. (Note 
that it is also analytic in the disk of radius $y+1$ centered at $z = y$.)   

Directly from the Cauchy integral formula (see, for example,~\cite[Lemma~2.2]{Alpert_91_01}), we know that 
\[
\left| \frac{\partial^n}{\partial y^n}\Lambda\left(\tfrac{x+y}{2}\right)\right| \leq \frac{n!}{y^n} \sup_{0\leq \theta\leq 2\pi} \left| \Lambda\left(\tfrac{ye^{i\theta}+y}{2}\right) \right|. 
\]
Since $| \Lambda( z ) | \leq e/(|z+1|)$ for all $z\in\mathbb{C}$ such that ${\rm Re}(z)\geq 0$~\cite[Lemma~2.4]{Alpert_91_01}, we have
\[
\left| \frac{\partial^n}{\partial y^n}\Lambda\left(\tfrac{x+y}{2}\right)\right| \leq \frac{n!}{y^n} \sup_{0\leq \theta\leq 2\pi}\frac{e}{\left| \tfrac{ye^{i\theta}+y}{2}+1\right|}\leq \frac{n!\, e}{y^n},
\]
as required.
\end{proof} 

Now that we have a bound on the $y$-derivatives of $\Lambda\left(\tfrac{x+y}{2}\right)$, we can derive 
error bounds on a Taylor expansion of $\Lambda\left(\tfrac{x+y}{2}\right)$ in the $y$-variable 
on the domain $[0,N]\times [\beta^kN,\beta^{k-1}N]$. 

\begin{lemma} 
Let $\epsilon>0$, $k$ be an integer, and $\tfrac{1}{2}<\beta <1$ such that $\beta^kN > 1$. Then, there is a rank $K=\lceil \log(e/\epsilon)/\log(\beta/(1-\beta))\rceil$ function 
$g_k(x,y)$ on $[0,N]\times [\beta^kN,\beta^{k-1}N]$ such that 
\[
\left|\Lambda\left(\tfrac{x+y}{2}\right) - g_k(x,y)\right|  \leq \epsilon,\qquad (x,y)\in [0,N]\times [\beta^kN,\beta^{k-1}N].
\]
\label{lem:LowNumericalRankSections}
\end{lemma} 
\begin{proof} 
By Taylor's theorem expanding about $y = \beta^{k}N$, we have
\begin{equation}
\Lambda\left(\frac{x+y}{2}\right)  =  \underbrace{\sum_{j=0}^{K-1} \frac{(y-\beta^{k}N)^j}{j!} \left( \frac{\partial^n}{\partial y^n}\Lambda\left(\tfrac{x+y}{2}\right) \bigg|_{y=\beta^kN}\right)}_{=g_k(x,y)} + R_{K-1}(x,y),\qquad x\in[0,N].
\label{eq:Taylor}
\end{equation} 
Moreover, for $(x,y) \in[0,N]\times [\beta^kN,\beta^{k-1}N]$ we have
\[
\left| R_{K-1}(x,y) \right| \leq \frac{\left(\beta^{k-1}N-\beta^kN\right)^K}{K!} \frac{K!\, e}{\left(\beta^kN\right)^K}\leq e\left(\frac{1-\beta}{\beta}\right)^K,
\]
where we used the fact that $y\in [\beta^kN,\beta^{k-1}N]$ and Lemma~\ref{lem:BoundedDerivative}.  Since $K = \lceil \log(e/\epsilon)/\log(\beta/(1-\beta))\rceil$,  the Taylor error satisfies $\left| R_{K-1}(x,y) \right| \leq \epsilon$.  The result follows since $g_k(x,y)$ in~\eqref{eq:Taylor} 
is a function of rank at most $K$, i.e., it can be written as a sum of $K$ terms of the form $c(x)r(y)$. 
\end{proof} 

Lemma~\ref{lem:LowNumericalRankSections} shows that on each domain of the form $[0,N]\times [\beta^kN,\beta^{k-1}N]$ the function 
$\Lambda\left(\tfrac{x+y}{2}\right)$ can be approximated, up to an error of $\epsilon$, by a rank $\mathcal{O}(\log(1/\epsilon))$ function. 
Importantly, the rank of the function $g_k(x,y)$ on $[0,N]\times [\beta^kN,\beta^{k-1}N]$ in Lemma~\ref{lem:LowNumericalRankSections} does 
not depend on $k$. This means that $\Lambda\left(\tfrac{x+y}{2}\right)$ has approximately the same numerical rank on 
each subdomain in Figure~\ref{fig:GeometricSubdivision}.  Since $[0,N]\times [1,N]$ was partitioned into a 
total of $\mathcal{O}(\log N)$ domains, the function $\Lambda\left(\tfrac{x+y}{2}\right)$ on $[0,N]\times [1,N]$ can be approximated
up to an error of $\epsilon$, by a rank $\mathcal{O}(\log N\log(1/\epsilon))$ function. 

\begin{theorem}
 There is a rank $K$ function $h$ defined on $[0,N]\times [1,N]$ such that
 \[
 \left|\Lambda\left(\frac{x+y}{2}\right) - h(x,y)\right|\leq \epsilon, \qquad (x,y)\in [0,N]\times [1,N],
 \]
 where $K = \lceil \log_{\tfrac{1}{\beta}}(N)\rceil \cdot \lceil \log(e/\epsilon)/\log(\beta/(1-\beta))\rceil = \mathcal{O}(\log N\log(1/\epsilon))$.
 \label{thm:LowNumericalRank} 
\end{theorem} 
\begin{proof}
Let $p = \lceil \log_{\tfrac{1}{\beta}}(N)\rceil$. 
Concatenate together the $\lceil \log_{\tfrac{1}{\beta}}(N)\rceil$ functions $g_1, \ldots, g_p$ on the domains 
$[0,N]\times [1,\beta^{k} N],  [0,N]\times [\beta^k N,\beta^{k-1} N], \ldots, [0,N]\times [\beta N,N]$, respectively,
from Lemma~\ref{lem:LowNumericalRankSections}.  The resulting function, say $h$, has a rank at most the sum 
of the ranks of $g_1,\ldots,g_p$. 
\end{proof} 
 
Theorem~\ref{thm:LowNumericalRank} is sufficient for the purposes of this paper. One can sample the constructed function 
in $h(x,y)$ in Theorem~\ref{thm:LowNumericalRank} at the tensor grid $\{0,\ldots,N\}\times \{1,\ldots,N\}$ to obtain a 
rank $\mathcal{O}( \log N \log(1/\epsilon ) )$ matrix $\hat{H}$ that approximates $\tilde{H}$ in~\eqref{eq:SplitFirstRow}, 
up to an accuracy of $\epsilon$. Therefore, $H$ can also be approximated by a rank $K = \mathcal{O}( \log N \log(1/\epsilon ) )$
matrix. In practice, we use the pivoted Cholesky algorithm (see Section~\ref{sec:PivotedCholesky}) to construct a low rank 
approximant for $H$ in $\mathcal{O}(K^2N)$ operations.  

Using the formula in~\eqref{eq:Cdecomposition} we can then calculate $\underline{c}^{cheb} = M\underline{c}^{leg}$ via a sum of $\mathcal{O}(K)$ 
diagonally-scaled Toeplitz matrix-vector products. Hence, we have described an $\mathcal{O}(N(\log N)^2)$ algorithm for 
computing $\underline{c}^{cheb} = M\underline{c}^{leg}$ (see Figure~\ref{fig:Summary}). 


\section{Numerical results for Legendre-to-Chebyshev conversion}\label{sec:NumericalResults} 
All the numerical results were performed on a 3.1 GHz Intel Core i7 Macbook Pro 2015 with MATLAB 2015b or Julia v0.4.5~\cite{Bezanson_14_01}. 
In these numerical experiments we employ three different algorithms for computing the matrix-vector product 
$\underline{c}^{cheb} = M \underline{c}^{leg}$: 
\begin{itemize} 
\item {\bf Direct:} The direct algorithm computes $\underline{c}^{cheb}$ by first 
constructing the $(N+1)\times (N+1)$ matrix $M$ one row at a time and then calculating the dot product with $\underline{c}^{leg}$. 
Therefore, the vector $\underline{c}^{cheb}$ is computed entry-by-entry, costing a total of $\mathcal{O}(N^2)$ operations and 
requiring $\mathcal{O}(N)$ storage. 
\item {\bf ASY:} The algorithm that for shorthand we call ASY here is described in~\cite{Hale_14_01}. It computes the matrix-vector product
$\underline{c}^{cheb} = M \underline{c}^{leg}$ in $\mathcal{O}(N(\log N)^2/\log\log N)$ operations by using a trigonometric 
asymptotic formula for Legendre polynomials. Before this paper, it was the algorithm employed in the
{\tt leg2cheb} command in Chebfun~\cite{Chebfun}.    
\item {\bf New:} The algorithm described in this paper. It is summarized in Figure~\ref{fig:Summary}, costing 
$\mathcal{O}(N(\log N)^2)$ operations. 
\end{itemize} 
Other algorithms for computing the matrix-vector product $\underline{c}^{cheb} = M\underline{c}^{leg}$ in fewer than $\mathcal{O}(N^2)$ 
operations are given in the pioneering paper by Alpert and Rokhlin~\cite{Alpert_91_01} as well as~\cite{Keiner_11_01,Orszag_86_01}. 

As a first test we take arbitrarily distributed vectors $\underline{c}^{leg}$ with various rates of decay and consider
the accuracy of our algorithm described in this paper against an extended precision computation (performed 
using the {\tt BigFloat} type in Julia).\footnote{In particular, the vector corresponding to, say, $N = 100$ with $\mathcal{O}(n^{-1})$
decay can be reproduced exactly by the Julia code {\tt srand(0); c = randn(101)./(1:101)}. The fixed random seed is employed for the sake of reproducibility.} Figure~\ref{fig:Leg2chebResults} (left) shows 
the absolute maximum errors in the computed vectors $\underline{c}^{cheb}$ for $10\leq N\leq 10^4$.  In~\cite[Fig.~5.1]{Hale_14_01}
analogous errors were calculated for the direct and ASY algorithms.  In Table~\ref{tab:ErrorGrowth} we summarise the observed error growth in the
absolute maximum for the three different algorithms. In many applications the Legendre expansion in~\eqref{eq:pN} represents a 
polynomial interpolant of a smooth function. In this setting, if the function is H\"{o}lder continuous with parameter greater than $0$, then 
we observe that our new algorithm has essentially no error growth with $N$. 

\begin{figure}
\begin{minipage}{.49\textwidth} 
\begin{overpic}[width=\textwidth]{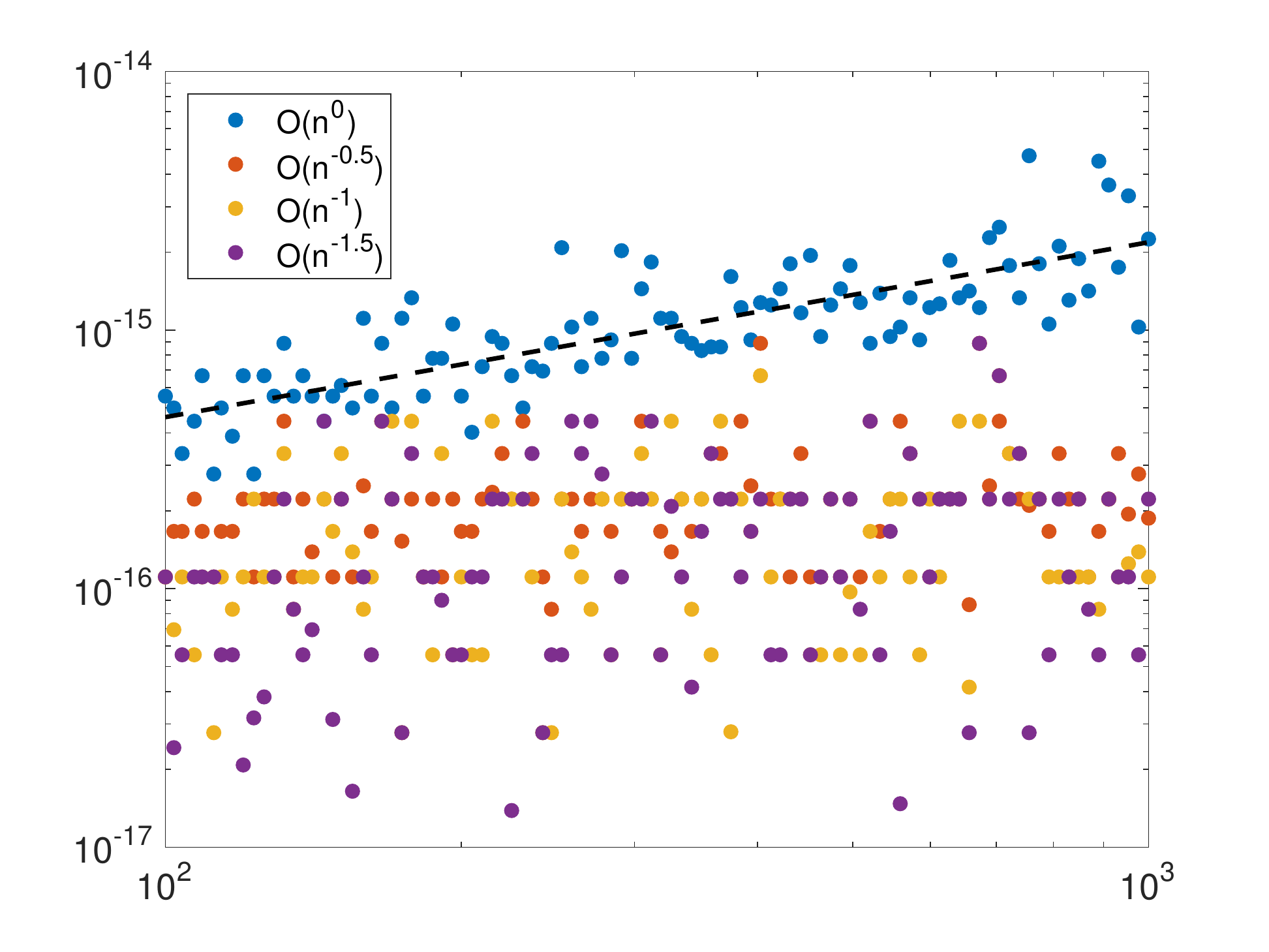}
\put(50,0) {$N$}
\put(0,22) {\rotatebox{90}{Absolute error}}
\put(35,52) {\rotatebox{10}{$\mathcal{O}(N^{0.5}\log N)$}}
\end{overpic} 
\end{minipage}
\begin{minipage}{.49\textwidth} 
\begin{overpic}[width=\textwidth]{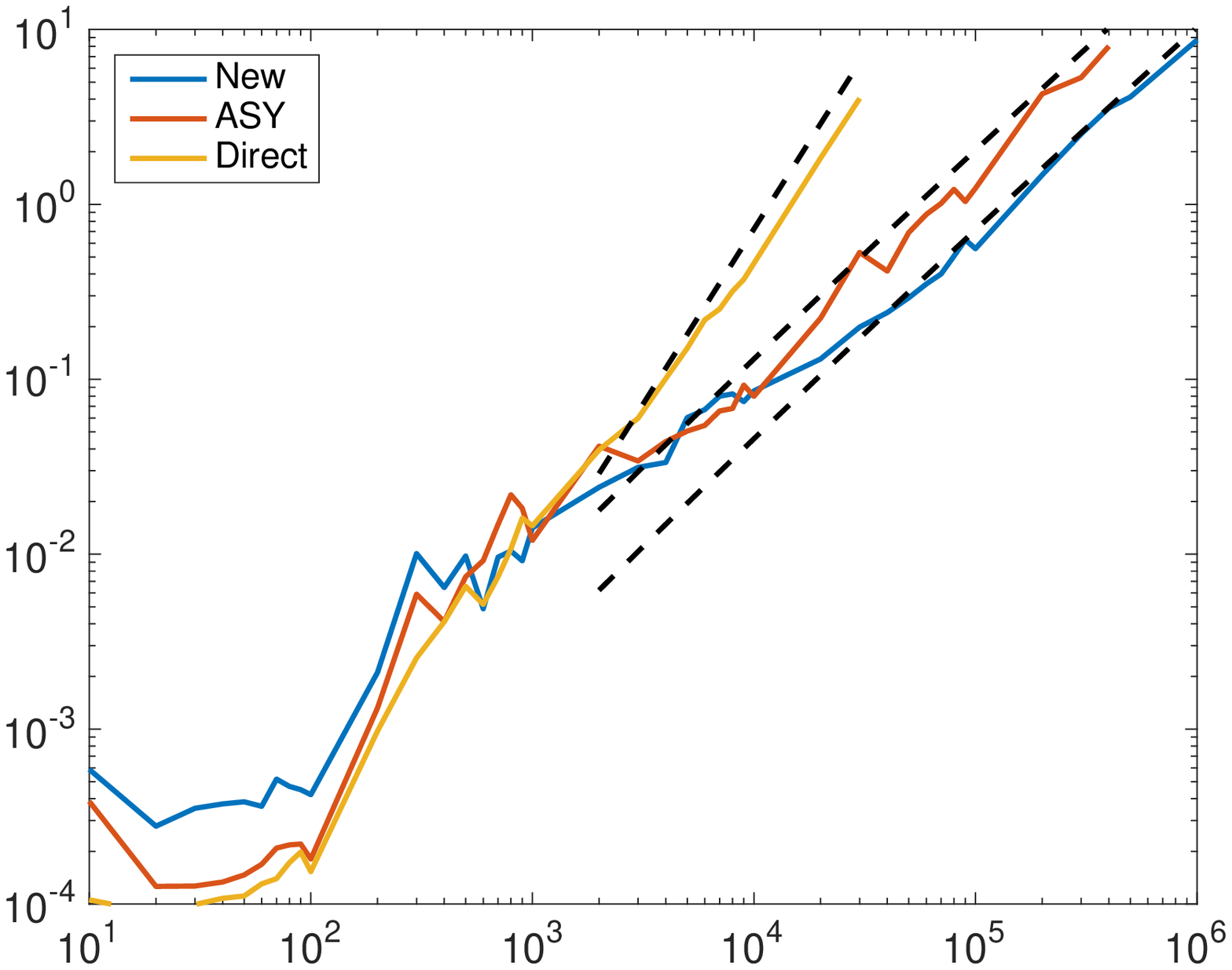}
\put(45,45) {\rotatebox{65}{$\mathcal{O}(N^2)$}} 
\put(60,33) {\rotatebox{45}{$\mathcal{O}(N(\log N)^2)$}} 
\put(0,10) {\rotatebox{90}{Execution time (seconds)}}
\put(50,0) {$N$}
\end{overpic} 
\end{minipage} 
\caption{Left: Errors in computing $\underline{c}^{cheb} = M\underline{c}^{leg}$ with various decay rates in $\underline{c}^{leg}$. A vector is created 
using {\tt randn(N)} in Julia and then the vector is scaled so the $n$th entry is $\mathcal{O}(n^0)$, $\mathcal{O}(n^{-0.5})$, $\mathcal{O}(n^{-1})$, $\mathcal{O}(n^{-1.5})$. The dashed line shows the observed error growth in the case when $\underline{c}^{leg}$ does not decay. 
Right: Execution times between the direct (yellow), ASY (red), and the new algorithm described in this paper (blue). The new algorithm 
presented in this paper is about $2$ or $3$ times faster than ASY, while being more accurate and only requiring a handful of lines of code to implement.}
\label{fig:Leg2chebResults} 
\end{figure} 

\begin{table}
\centering
\begin{tabular}{ccccc}
\multicolumn{5}{c}{Decay of the vector $\underline{c}^{leg}$}  \\[5pt]
   & $\mathcal{O}(1)$ &  $\mathcal{O}(n^{-0.5})$ & $\mathcal{O}(n^{-1})$ & $\mathcal{O}(n^{-1.5})$\\[5pt] 
\hline
\rule{0pt}{1.5\normalbaselineskip}Direct & $\mathcal{O}(N^{1.5}/\log N)$ &  $\mathcal{O}(N/\log N)$ &  $\mathcal{O}(N^{0.5}/\log N)$ &  $\mathcal{O}(\log N)$ \\[5pt]
ASY  & $\mathcal{O}(N)$ & $\mathcal{O}(N^{0.5})$ & $\mathcal{O}(\log N)$ & $\mathcal{O}(1)$\\[5pt]
New  & $\mathcal{O}(N^{0.5}\log N)$ & $\mathcal{O}(1)$ & $\mathcal{O}(1)$ & $\mathcal{O}(1)$\\[5pt]
\hline
\end{tabular} 
\caption{Observed absolute maximum error growth in computing $\smash{\underline{c}^{cheb} = M\underline{c}^{leg}}$ when $\underline{c}^{leg}$ has different decay rates. 
The error growth of the ASY algorithm  given in~\cite[Fig.~5.1]{Hale_14_01} were computed using the implementation of {\tt leg2cheb} found in 
version 5.3 of Chebfun~\cite{Chebfun}. These are worse than those presented here because the ASY algorithm now uses Reinsch's modification of the Legendre recurrence as recommended in~\cite{slevinsky2016use}. We observe that if the Legendre expansion in~\eqref{eq:pN} comes from a polynomial interpolant of a H\"older continuous function with parameter $>0$, then our new algorithm has essentially no error growth.} 
\label{tab:ErrorGrowth}
\end{table} 

For a second test, in Figure~\ref{fig:Leg2chebResults} (right) we compare the execution times for the three algorithms.  
Despite the direct algorithm requiring $\mathcal{O}(N^2)$ operations, 
it is computationally more efficient when $N\leq 512$. The new algorithm presented here is $2$ or $3$ times faster than the ASY algorithm for large 
$N$, while being conceptually simpler and more accurate for nondecaying vectors $\underline{c}^{leg}$. Based on these numerical experiments, the {\tt leg2cheb} command in Chebfun~\cite{Chebfun} and the 
{\tt leg2cheb} in FastTransforms.jl~\cite{FastTransforms} use the direct algorithm when $N\leq 512$ and the new algorithm otherwise.

\section{Other polynomial basis conversions}\label{sec:OtherConversions}
So far the paper has focused on the task of converting Legendre coefficients for $p_N$ in~\eqref{eq:pN} to Chebyshev 
coefficients.  In this section we consider other standard polynomial basis conversions, showing how our $\mathcal{O}(N(\log N)^2)$ 
algorithm summarized in Figure~\ref{fig:Summary} remains applicable. 

\subsection{Chebyshev-to-Legendre conversion}\label{sec:cheb2leg}
To compute the Legendre coefficients of a given polynomial $p_N$ in fewer than $\mathcal{O}(N^2)$ operations, one can first compute the Chebyshev coefficients using the discrete cosine transform (DCT) of its values at Chebyshev points in $\mathcal{O}(N\log N)$ operations~\cite{Gentleman_72_01}, then use a fast Chebyshev-to-Legendre conversion. Alternatively, a direct transform taking values of the polynomial in the complex plane to Legendre coefficients is given in~\cite{Iserles_11_01} and 
a fast transform for converting values of the polynomial at Legendre points to Legendre coefficients is given in~\cite{Hale_15_01}. 

The inverse of the Legendre-to-Chebyshev matrix $M^{-1}$, denoted by $L$ in~\cite{Alpert_91_01}, 
converts Chebyshev coefficients to Legendre coefficients, i.e., $\underline{c}^{leg} = L\underline{c}^{cheb}$ in~\eqref{eq:pN}.
Explicit formulas for the entries of $L=M^{-1}$ are given as follows~\cite[(2.19)]{Alpert_91_01}:
\begin{equation}
L_{jk} = \begin{cases} 1, & j= k = 0, \\ \frac{\sqrt{\pi}}{2\Lambda(j)}, & 0< j = k \leq N,\\ -k(j+\tfrac{1}{2})\left(\frac{\Lambda\left(\frac{k-j-2}{2}\right)}{k-j}\right)\left(\frac{\Lambda\left(\frac{j+k-1}{2}\right)}{j+k+1}\right), & 0 \leq j < k \leq N, \text{ } k-j \text{ even}, \end{cases}
\label{eq:FormulaL}
\end{equation} 
where $\Lambda(z) = \Gamma(z+1/2)/\Gamma(z+1)$ and $\Gamma(z)$ is the gamma function. The fact that $L$ can 
be written as $L = D_1(T\circ H)D_2$ is almost immediate from~\eqref{eq:FormulaL}. In particular, we have 
$D_1 = \mathrm{diag}(\frac12,\frac32,\frac52,\ldots,\tfrac{2N+1}{2})$, $D_2 = -\frac14\mathrm{diag}(\frac{4}{\sqrt{\pi}},1,2,3,\ldots,N)$, and
\[
\begin{aligned}
T_{jk} &= \begin{cases} \Gamma\!\left(\frac{k-j-1}{2}\right)/\Gamma\!\left(\frac{k-j+2}{2}\right), & 0\leq j\leq k\leq N, \text{ }k-j \text{ even},\\ 0, & \text{otherwise},\end{cases} \\
H_{jk} &=  \begin{cases} \Gamma\!\left(\frac{j+k}{2}\right)/\Gamma\!\left(\frac{j+k+3}{2}\right), & 0\leq j, k\leq N, \text{ } j+k > 0,\\ 1, & j=k=0.\end{cases}
\end{aligned}
\]

Unfortunately, the matrix $H$ is not positive semidefinite. This turns out not to matter, because the submatrix $\tilde{H} = \left(H_{jk}\right)_{1\leq j,k\leq N}$ is positive semi-definite by the following identity (see Lemma~\ref{lem:hamburger}),
\[
H_{jk} = \frac{4}{\sqrt{\pi}}\int_0^1 x^{j+k-1}(1-x^2)^{1/2} dx, \quad 1\leq j,k\leq N,
\]
Moreover, a similar analysis to that in Section~\ref{sec:HankelLowRank} can be used to show that $\tilde{H}$ can be approximated, up to an error of $\epsilon>0$, by a rank $\mathcal{O}(\log N \log(1/\epsilon))$ matrix. Hence, when computing $\underline{c}^{leg} = L\underline{c}^{cheb}$, we compute the first entry of $\underline{c}^{leg}$ directly, and use the algorithm described in Figure \ref{fig:Summary} on $\left(L_{jk}\right)_{1\leq j,k\leq N}$ to compute the remaining entries. The resulting algorithm is implemented in the {\tt cheb2leg} commands in Chebfun~\cite{Chebfun} and FastTransforms.jl~\cite{FastTransforms}.

In Figure~\ref{fig:Cheb2legResults} we repeat the same experiments as for the Legendre-to-Chebyshev conversion in Section~\ref{sec:NumericalResults}. In Figure~\ref{fig:Cheb2legResults}
(left) we compute the maximum error of the resulting vector $\underline{c}^{leg}$ for different decay rates in $\underline{c}^{cheb}$. 
Due to the $\mathcal{O}(N)$ growth in the entries of $D_1$, we find that the conversion requires a decay faster than $\mathcal{O}(n^{-1})$ 
in $\underline{c}^{cheb}$ to have essentially no error growth. This holds when $p_N$ in~\eqref{eq:pN} is a 
Chebyshev interpolant of a H\"{o}lder continuous function with H\"{o}lder parameter $>1/2$. The observed error growth 
is less than that observed for the Chebyshev-to-Legendre algorithm in~\cite[Fig.~5.2]{Hale_14_01}. 

In Figure~\ref{fig:Cheb2legResults} (right) we show the execution times of the three algorithms: (1) Direct, an algorithm that costs 
$\mathcal{O}(N^2)$ operations and requires $\mathcal{O}(N)$ memory based on generating the whole matrix $L$ one row at a time, (2) 
ASY, an $\mathcal{O}(N(\log N)^2/\log\log N)$ complexity algorithm described in~\cite[Sec.~4]{Hale_14_01}, and (3) New, the algorithm 
described in this paper (see Figure~\ref{fig:Summary}).  Our new algorithm is faster than direct when $N\geq 1,\!000$ and is about $2$ or 
$3$ times faster than the ASY algorithm for large $N$.   

\begin{figure}
\begin{minipage}{.49\textwidth} 
\begin{overpic}[width=\textwidth]{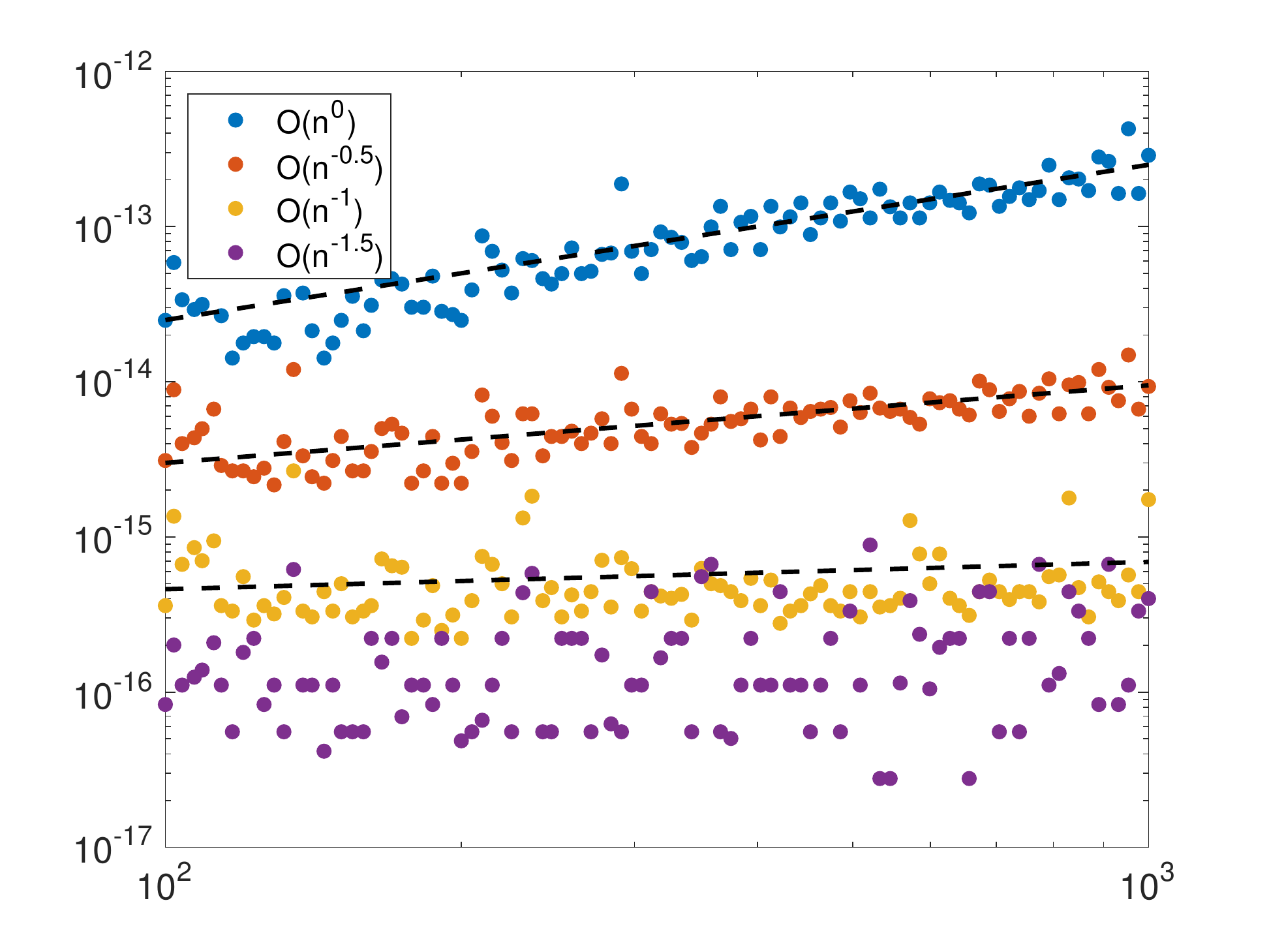}
\put(50,0) {$N$}
\put(0,22) {\rotatebox{90}{Absolute error}}
\put(45,57) {\rotatebox{14}{$\mathcal{O}(N)$}}
\put(45,43) {\rotatebox{8}{$\mathcal{O}(\sqrt{N})$}}
\put(45,31) {\rotatebox{3}{$\mathcal{O}(\log N)$}}
\end{overpic} 
\end{minipage}
\begin{minipage}{.49\textwidth} 
\begin{overpic}[width=\textwidth]{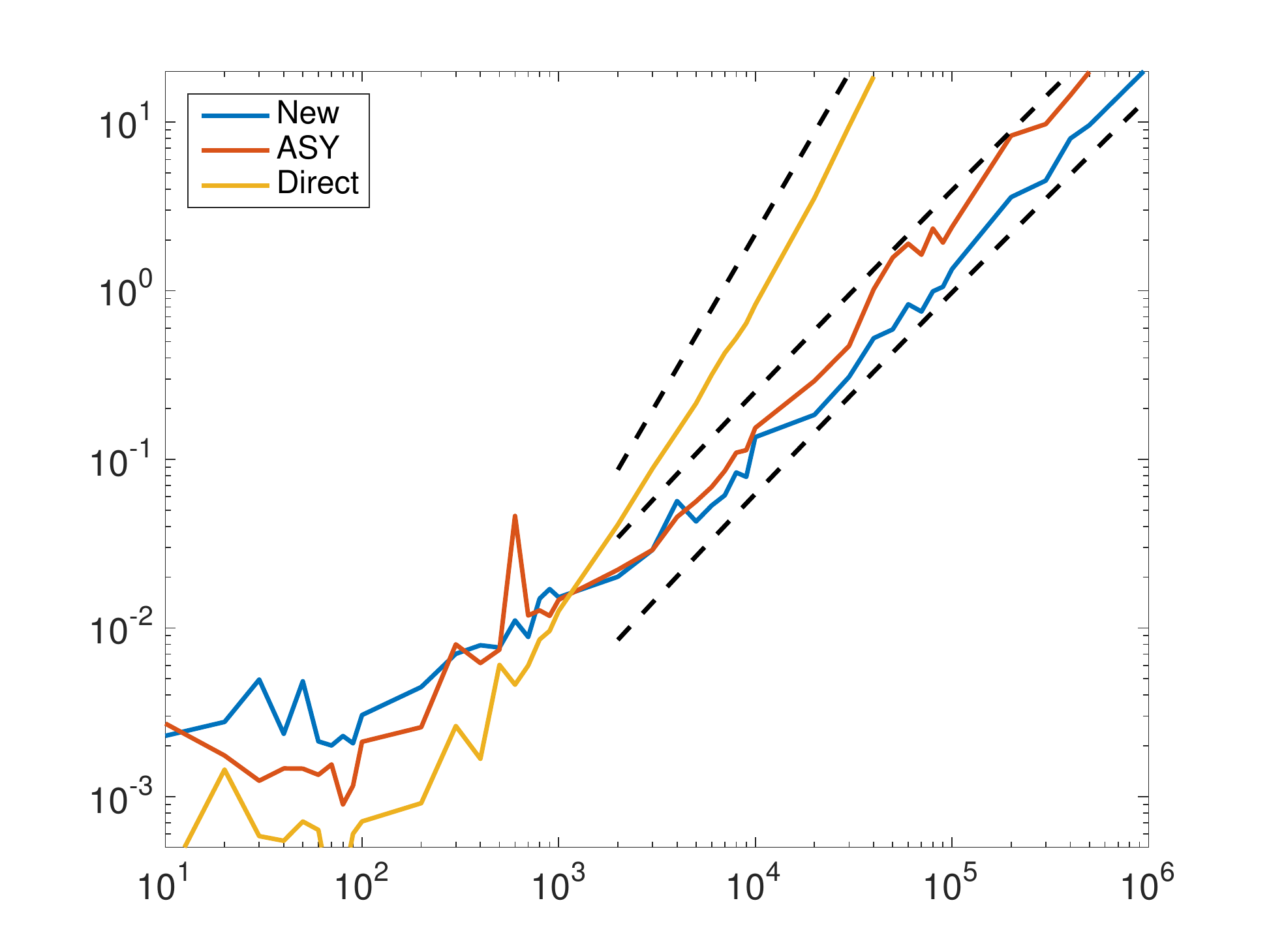}
\put(45,45) {\rotatebox{65}{$\mathcal{O}(N^2)$}} 
\put(60,30) {\rotatebox{45}{$\mathcal{O}(N(\log N)^2)$}} 
\put(0,10) {\rotatebox{90}{Execution time (seconds)}}
\put(50,0) {$N$}
\end{overpic} 
\end{minipage} 
\caption{Left: Errors in computing $\underline{c}^{leg} = L\underline{c}^{cheb}$ with various decay rates in $\underline{c}^{cheb}$. Arbitrary vectors are 
generated in precisely the same manner as in Section~\ref{sec:NumericalResults}. The dashed lines show the observed error growth for the various decay rates. 
Right: Execution times between the direct (yellow), ASY (red), and the new algorithm described in this paper (blue).}
\label{fig:Cheb2legResults}
\end{figure} 

\subsection{Ultraspherical-to-ultraspherical conversion}\label{sec:UltraphericalConversions} 
The ultraspherical polynomial of degree $k$ and parameter $\lambda>0$ is denoted by $C_k^{(\lambda)}$~\cite[Tab.~18.3.1]{NISTHandbook}. 
If $\lambda>0$, then $\{C_0^{(\lambda)},C_1^{(\lambda)}\ldots,\}$ is a family of orthogonal polynomials 
that are orthogonal with respect to the weight function $(1-x^2)^{\lambda-1/2}$ on $[-1,1]$.  Ultraspherical 
polynomials can be seen as a generalization of Legendre polynomials since $P_k(x) = C^{(1/2)}_k(x)$~\cite[(18.7.9)]{NISTHandbook}. 
For papers on computing and converting ultraspherical expansions see, for example,~\cite{Cantero_12_01,Keiner_09_01}.   

Let $\lambda_1>0$ and $\lambda_2>0$. The degree $N$ polynomial $p_N$ in~\eqref{eq:pN} can be expanded in the two 
ultraspherical polynomial bases associated to $\lambda_1$ and $\lambda_2$, i.e., 
\[
p_n(x) = \sum_{k=0}^N c_k^{(\lambda_1)} C_k^{(\lambda_1)}(x)  = \sum_{k=0}^N c_k^{(\lambda_2)} C_k^{(\lambda_2)}(x).
\]
There is an upper-triangular conversion matrix, $A$, such that $\underline{c}^{(\lambda_2)} = A\underline{c}^{(\lambda_1)}$. We desire a fast 
algorithm for computing the matrix-vector product  $\underline{c}^{(\lambda_2)} = A\underline{c}^{(\lambda_1)}$. 
There are several cases to consider: (1) $\lambda_1-\lambda_2$ is an integer, (2) $|\lambda_1-\lambda_2|<1$, and (3) $|\lambda_1-\lambda_2|>1$, but 
the difference is a noninteger.  

\subsubsection*{Case 1: $\mathbf{\lambda_1-\lambda_2}$ is an integer}
 If $\lambda_1-\lambda_2$ is an integer, then an $\mathcal{O}(N)$ matrix-vector product is immediate from the recurrence relation~\cite[(18.9.7)]{NISTHandbook}. For example, if $\lambda_2-\lambda_1$ is a positive integer, then we can factor $A$ into the product of sparse 
matrices as follows: 
\begin{equation} 
A = P_NS_{\lambda_2-1} \cdots S_{\lambda_1+1}S_{\lambda_1}P_N^T, \qquad S_{\lambda} = \begin{pmatrix}1 & & -\frac{\lambda}{\lambda+2}\\ & \frac{\lambda}{\lambda+1} & & -\frac{\lambda}{\lambda+3}\\ & & \frac{\lambda}{\lambda+2} & & -\frac{\lambda}{\lambda+4}\\ & & & \ddots & & \ddots \end{pmatrix}. 
\label{eq:sparseFactorization}
\end{equation} 
where $P_N = \begin{bmatrix}I_{N+1} & O \end{bmatrix}$ is the $(N+1)\times \infty$ project matrix. 
The matrix-vector product $\underline{c}^{(\lambda_2)} = A\underline{c}^{(\lambda_1)}$ can be computed in $\mathcal{O}(N)$ operations by applying each truncated sparse factor in turn. 
Since $S_{\lambda}$ is banded and upper-triangular for all $\lambda>0$, the matrix-vector product 
$\underline{c}^{(\lambda_2)} = A\underline{c}^{(\lambda_1)}$ can also be computed in $\mathcal{O}(N)$ operations when $\lambda_2<\lambda_1$ by using backward substitution. 
The factorization in~\eqref{eq:sparseFactorization} is one key decomposition for the ultraspherical spectral method~\cite{Olver_13_01}. 

\subsubsection*{Case 2: $\mathbf{|\lambda_1-\lambda_2|<1}$}
Now assume that $|\lambda_1 - \lambda_2|<1$, then the conversion matrix, $A$, in $\underline{c}^{(\lambda_2)} = A\underline{c}^{(\lambda_1)}$ 
has the following explicit formula~\cite[(3.6)]{Keiner_09_01}:
\begin{equation}
A_{jk} = \begin{cases} \frac{\Gamma(\lambda_2)(j + \lambda_2)}{\Gamma(\lambda_1)\Gamma(\lambda_1-\lambda_2)} \left(\frac{\Gamma\left(\frac{k-j}{2} + \lambda_1-\lambda_2\right)}{\Gamma\left(\frac{k-j}{2}+1\right)}\right)\left(\frac{\Gamma\left(\frac{k+j}{2} + \lambda_1\right)}{\Gamma\left(\frac{k+j}{2}+\lambda_2+1\right)}\right), & 0\leq j\leq k, \text{ } k-j \text{ even},\\
0,&\text{otherwise}.\end{cases} 
\label{eq:UltraSconversion}
\end{equation} 
The formula in~\eqref{eq:UltraSconversion} reveals that the matrix $A$ can be written as a diagonally-scaled Toeplitz-dot-Hankel matrix.  More precisely, 
let $D_1 = {\rm diag}\left(\lambda_2,\lambda_2+1,\ldots,\lambda_2+N\right)$, $D_2 = I_{N+1}$ be the $(N+1)\times (N+1)$ identity matrix, and 
\[
\begin{aligned}
T_{jk} &= \begin{cases} \frac{\Gamma(\lambda_1-\lambda_2)\Gamma\left(\frac{k-j}{2} + \lambda_1-\lambda_2\right)}{\Gamma\left(\frac{k-j}{2}+1\right)},& 0\leq j\leq k\leq N, \text{ } k-j \text{ even},\\ 0,& \text{otherwise},\end{cases} \\
H_{jk} &= \frac{\Gamma(\lambda_2)\Gamma\left(\frac{k+j}{2} + \lambda_1\right)}{\Gamma(\lambda_1)\Gamma\left(\frac{k+j}{2}+\lambda_2+1\right)}, \quad 0\leq j,k\leq N.
\end{aligned}
\]
Since the entries of the Hankel part can be expressed as 
\[
H_{jk} = \frac{2\Gamma(\lambda_2)}{\Gamma(\lambda_1) \Gamma(\lambda_2-\lambda_1+1)}\int_0^1 x^{j+k} x^{2(\lambda_1 - 1/2)} (1-x^2)^{\lambda_2-\lambda_1} \mathrm{d} x, \quad 0\leq j,k\leq N,
\]
we know from Lemma~\ref{lem:hamburger} that $H$ is real, symmetric, and positive semidefinite (note that $\lambda_2 - \lambda_1> - 1$ and $2(\lambda_1-1/2) > -1$ so that the measure $\mu_H$ is locally finite). 
Moreover, a similar analysis to Section~\ref{sec:HankelLowRank} shows that $H$ can be approximated, up to an error of $\epsilon$, by a rank $K = \mathcal{O}(\log N\log(1/\epsilon))$ matrix.   Therefore, the algorithm that is summarized in Figure~\ref{fig:Summary} is applicable in this case and can be used to  compute $\underline{c}^{(\lambda_2)} = A\underline{c}^{(\lambda_1)}$ in $\mathcal{O}(N(\log N)^2)$ operations.  

\subsubsection*{Case 3: $\mathbf{|\lambda_1-\lambda_2|>1}$, but the difference is a noninteger}
If $|\lambda_1-\lambda_2|>1$, then we reduce the quantity $|\lambda_1-\lambda_2|$ by converting to either increase or reduce 
$\lambda_1$ by one (see Case $1$).  This is repeated if necessary until $|\lambda_1-\lambda_2|<1$ and the criterion for Case $2$ 
is satisfied. 

%
%

\bigskip 

Conversions such as ultraspherical-to-Chebyshev and Chebyshev-to-ultraspherical are  associated to upper-triangular 
matrices that can also be written as diagonally-scaled Toeplitz-dot-Hankel matrices. Fast $\mathcal{O}(N(\log N)^2)$ 
algorithms based on Figure~\ref{fig:Summary} for these conversions are also possible using the formulas in~\cite[Lem.~3.1]{Keiner_09_01}.  

\subsection{Jacobi-to-Jacobi conversion}\label{sec:JacobiConversions} 
The Jacobi polynomial of degree $k$ and parameter $(\alpha,\beta)$ is denoted by $P_k^{(\alpha,\beta)}$, where 
$\alpha,\beta>-1$~\cite[Tab.~18.3.1]{NISTHandbook}. The family of orthogonal polynomials $\{P_0^{(\alpha,\beta)},P_1^{(\alpha,\beta)},\ldots,\}$ 
is orthogonal with respect to the weight function $(1-x)^\alpha(1+x)^\beta$ on $[-1,1]$. The Jacobi polynomials can be seen as 
a generalization of Chebyshev, Legendre, and ultraspherical polynomials, see~\cite[Sec.~18.7]{NISTHandbook}. 

For $\alpha,\beta,\gamma,\delta>-1$, the degree $N$ polynomial $p_N(x)$ in~\eqref{eq:pN} can be expanded in Jacobi bases as follows: 
\begin{equation}
p_N(x) = \sum_{k=0}^N c_k^{(\alpha,\beta)}P_k^{(\alpha,\beta)}(x) = \sum_{k=0}^N c_k^{(\gamma,\delta)}P_k^{(\gamma,\delta)}(x),
\label{eq:JacobiExpansions} 
\end{equation} 
where there is an upper-triangular matrix, $A$, such that $\underline{c}^{(\gamma,\delta)} = A\underline{c}^{(\alpha,\beta)}$. 
By the reflection formula $P_k^{(\alpha,\beta)}(x) = (-1)^kP_k^{(\beta,\alpha)}(x)$~\cite[Tab.~18.6.1]{NISTHandbook}, it is sufficient to assume
that $\beta = \delta$ in~\eqref{eq:JacobiExpansions}.\footnote{If $\beta\neq \delta$, then a Jacobi $(\alpha,\beta)$ expansion can first be 
converted to a Jacobi $(\alpha,\delta)$ expansion and then a Jacobi $(\gamma,\delta)$ expansion.} 
A different fast algorithm using off-diagonal low rank structure of the conversion matrix, which has 
a fast quasilinear complexity online cost and an $\mathcal{O}(N^2)$ precomputation is given in~\cite{Shen_16_01}. Another fast algorithm 
for computing Jacobi expansions coefficients of analytic functions is described in~\cite{Wang_14_01}.

As in Section~\ref{sec:UltraphericalConversions} there are several cases to consider when computing 
$\underline{c}^{(\gamma,\delta)} = A\underline{c}^{(\alpha,\beta)}$, where $\beta = \delta$, in fewer than $\mathcal{O}(N^2)$ operations: 
(1) $\alpha-\gamma$ is an integer, 
(2) $|\alpha-\gamma|<1$ and $\alpha+\beta>-1$, 
(3) $|\alpha-\gamma|<1$ and $\alpha+\beta\leq -1$, and
(4) $|\alpha-\gamma|>1$, but the difference is a noninteger. 

\subsubsection*{Case 1: $\mathbf{\alpha-\gamma}$ is an integer}
First, suppose that $\alpha - \gamma$ is an integer. A fast matrix-vector product for  $\underline{c}^{(\gamma,\beta)} = A\underline{c}^{(\alpha,\beta)}$
is almost immediate via the recurrence relation~\cite[18.9.5]{NISTHandbook} and~\cite[Tab.~18.6.1]{NISTHandbook}. That is, assuming that $\gamma>\alpha$ we can factor $A$ as follows: 
\[
A = P_NS_{(\gamma-1,\beta)}\cdots S_{(\alpha,\beta)}P_N^T,\quad S_{(\alpha,\beta)} = \begin{pmatrix} 1 & -\frac{\alpha+1}{\alpha+\beta+3}\\[3pt] & \frac{\alpha+\beta+2}{\alpha+\beta+3}&-\frac{\alpha+2}{\alpha+\beta+5}\\[3pt] &&\frac{\alpha+\beta+3}{\alpha+\beta+5}&-\frac{\alpha+3}{\alpha+\beta+7}\\[3pt] &&&\ddots &\ddots \end{pmatrix}, 
\] 
where $P_N = \begin{bmatrix}I_{N+1} & 0 \end{bmatrix}$ is the $(N+1)\times \infty$ project matrix. The matrix-vector product 
$\underline{c}^{(\gamma,\beta)} = A\underline{c}^{(\alpha,\beta)}$ can be computed in $\mathcal{O}(N)$ operations by applying each truncated sparse factor in turn. If $\gamma<\alpha$, then since each $S_{(\alpha,\beta)}$ is bidiagonal and upper-triangular $\underline{c}^{(\gamma,\beta)} = A\underline{c}^{(\alpha,\beta)}$ can still be computed in $\mathcal{O}(N)$ operations by using backward substitution. 

\subsubsection*{Case 2: $\mathbf{|\alpha-\gamma|<1}$ and $\mathbf{\alpha+\beta}>-1$}
When $|\alpha-\gamma|<1$ there is no known \emph{sparse} factorization for the conversion matrix like in Case 1. However, the 
following explicit formula for its entries is known~\cite{Keiner_11_01,Maroni_08_01}: 
\begin{multline*}
 A_{jk} = \frac{(2j+\gamma +\beta + 1)}{\Gamma(\alpha-\gamma)}\frac{\Gamma(k+\beta+1)}{\Gamma(k+\alpha+\beta+1)}\frac{\Gamma(j + \gamma+\beta+1)}{\Gamma(j+\beta+1)} \\ 
 \times\frac{\Gamma(k-j+\alpha-\gamma)}{\Gamma(k-j+1)} \frac{\Gamma(k+j + \alpha+\beta +1)}{\Gamma(k+j+\gamma+\beta+2)},
\end{multline*}
where $0\leq j\leq k\leq N$. The entries of $A$ are zero otherwise.  A careful inspection of this formula 
reveals that it can also be expressed as a diagonally-scaled Toeplitz-dot-Hankel matrix.  

Moreover, the entries of the Hankel matrix can be expressed as follows: 
\begin{equation}
 H_{jk} = \frac{1}{\Gamma(\gamma-\alpha + 1)} \int_0^1 x^{j+k} x^{\alpha+\beta} (1-x)^{\gamma-\alpha} \mathrm{d}x,\quad 0\leq j,k\leq N,
 \label{eq:HankelPosDefJacobi}
\end{equation}
proving that $H$ is real, symmetric, and positive semidefinite (see Lemma~\ref{lem:hamburger}) 
since $\alpha+\beta>-1$ and $\gamma-\alpha>-1$. A similar analysis to that in Section~\ref{sec:HankelLowRank}
shows that $H$ can be approximated, up to an error of $\epsilon$, by a rank $\mathcal{O}(\log N\log(1/\epsilon))$ matrix. Therefore, the 
$\mathcal{O}(N(\log N)^2)$ complexity algorithm summarized in Figure~\ref{fig:Summary} can be employed for Jacobi conversion in this case. 

The {\tt jac2jac} algorithm in Chebfun and FastTransforms.jl implements this and the other three cases.  In Figure~\ref{fig:jac2jacResults} (left) 
we test the accuracy of our algorithm by using it to compute $\smash{\underline{c}^{(-1/4,\sqrt{2}/2)} = A\underline{c}^{(0,\sqrt{2}/2)}}$
for various decay rates in the vector $\smash{\underline{c}^{(0,\sqrt{2}/2)}}$.  In Figure~\ref{fig:jac2jacResults} (right) we compare the 
execution time of this algorithm and a direct approach. We observe that our algorithm is faster in this case when $N \geq 512$. 

\begin{figure} 
\begin{minipage}{.49\textwidth} 
\begin{overpic}[width=\textwidth]{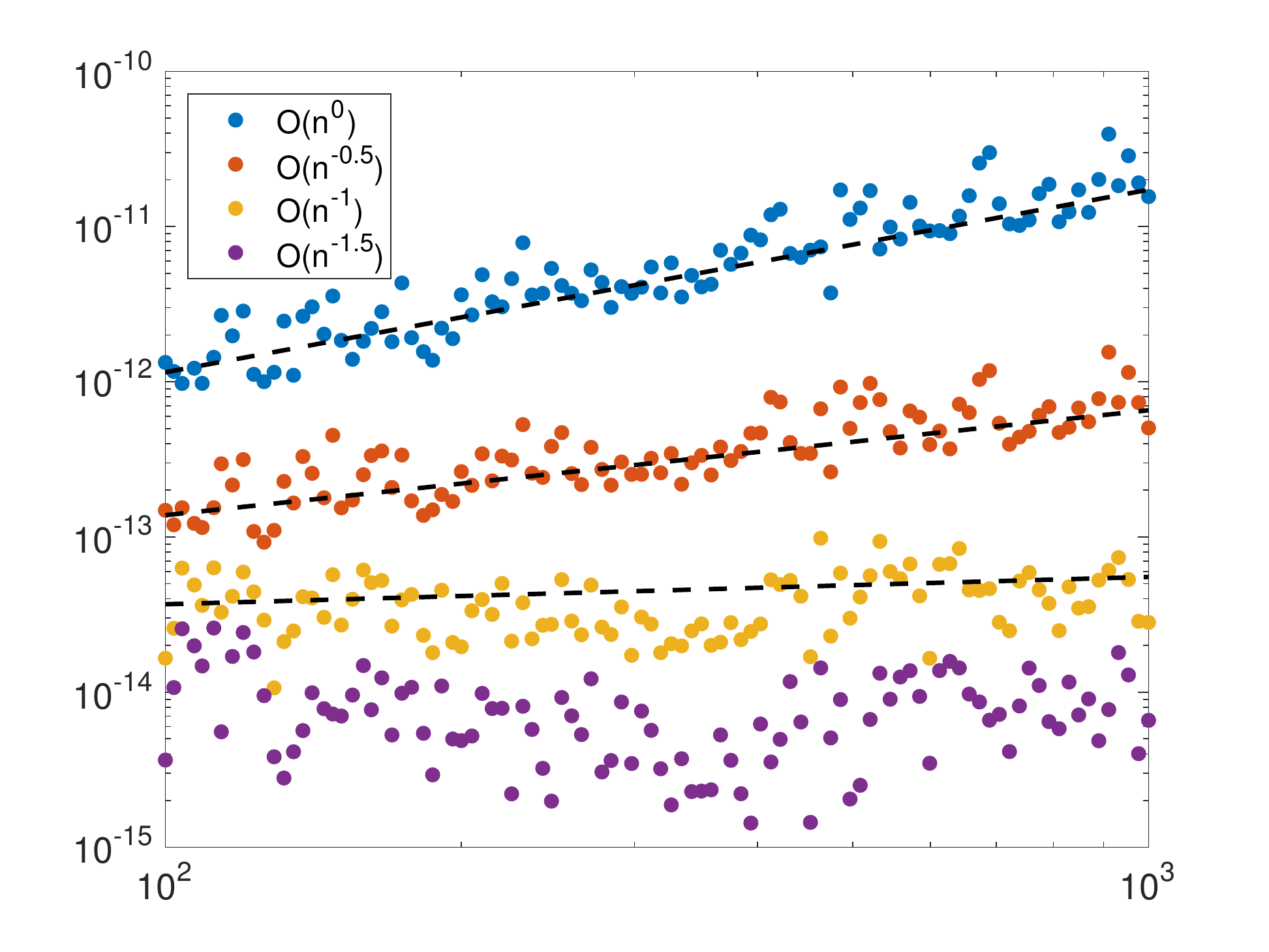}
\put(50,0) {$N$}
\put(0,22) {\rotatebox{90}{Absolute error}}
\put(45,57) {\rotatebox{14}{$\mathcal{O}(N\log N)$}}
\put(45,43) {\rotatebox{8}{$\mathcal{O}(\sqrt{N}\log N)$}}
\put(45,31) {\rotatebox{3}{$\mathcal{O}(\log N)$}}
\end{overpic} 
\end{minipage}
\begin{minipage}{.49\textwidth}
\begin{overpic}[width=\textwidth]{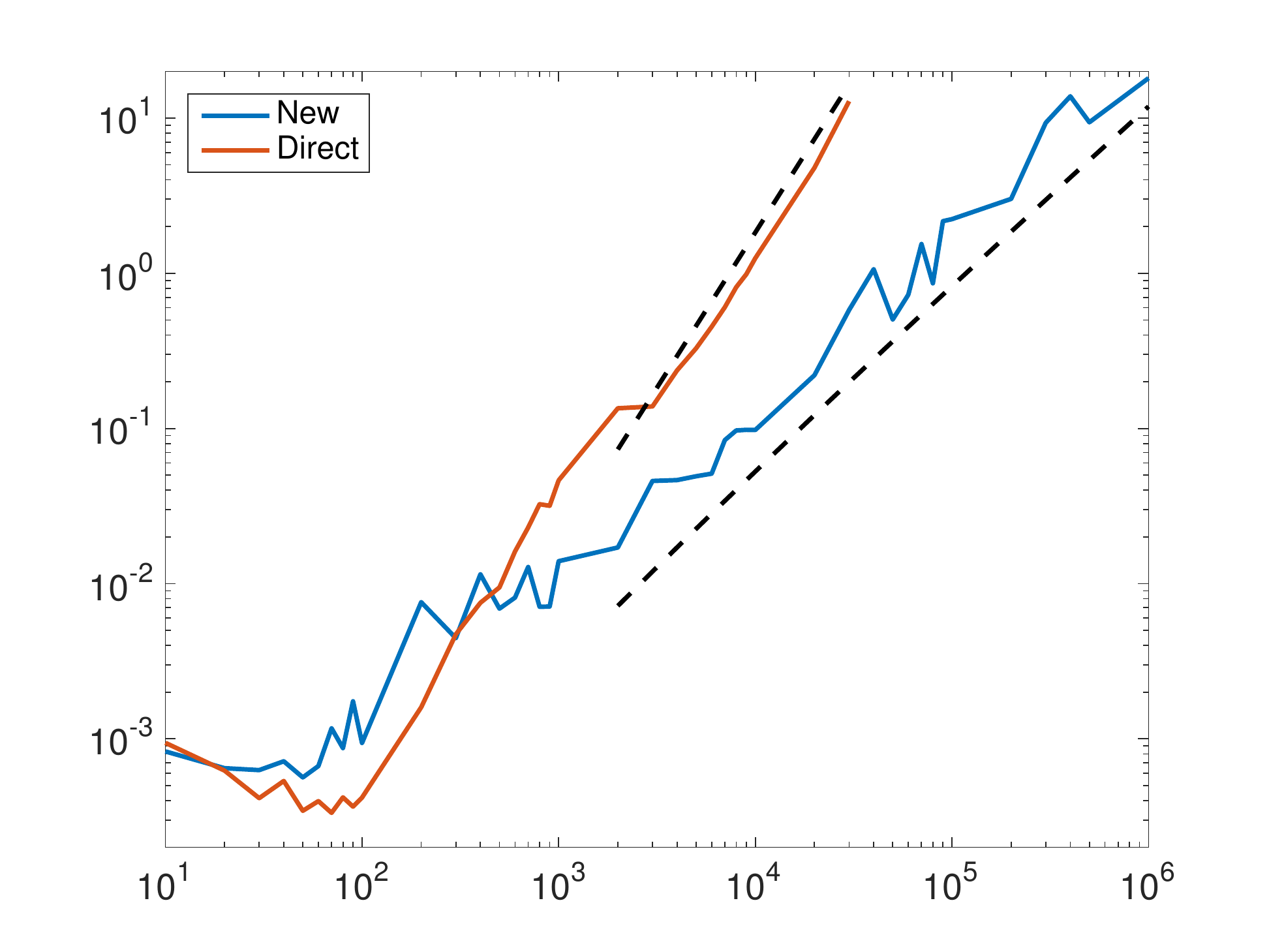}
\put(45,45) {\rotatebox{65}{$\mathcal{O}(N^2)$}} 
\put(60,30) {\rotatebox{45}{$\mathcal{O}(N(\log N)^2)$}} 
\put(0,10) {\rotatebox{90}{Execution time (seconds)}}
\put(50,0) {$N$}
\end{overpic} 
\end{minipage} 
\caption{Left: Errors in computing $\smash{\underline{c}^{(-1/4,\sqrt{2}/2)} = A\underline{c}^{(0,\sqrt{2}/2)}}$ with various decay rates in $\smash{\underline{c}^{(0,\sqrt{2}/2)}}$. Arbitrary vectors are generated in precisely the same manner as in sections~\ref{sec:NumericalResults} and Section~\ref{sec:cheb2leg}. The dashed lines show the observed error growth for the various decay rates. 
Right: Execution times between the direct (yellow), ASY (red), and the new algorithm described in this paper (blue).}
\label{fig:jac2jacResults}
\end{figure} 

\subsubsection*{Case 3: $\mathbf{|\alpha-\gamma|<1}$ and $\mathbf{\alpha+\beta}\leq -1$}
This is a situation where the Jacobi conversion matrix in $\underline{c}^{(\gamma,\beta)} = A\underline{c}^{(\alpha,\beta)}$ can be written as a diagonally-scaled Toeplitz-dot-Hankel matrix; however, the Hankel part is not positive semidefinite (see~\eqref{eq:HankelPosDefJacobi}). This is similar to what happens in Section \ref{sec:cheb2leg}. Indeed, the submatrix $\tilde{H} = \left(H_{jk}\right)_{1\leq j,k\leq N}$ is in fact positive semi-definite because $\alpha + \beta + 1 > -1$. Hence, we can do the same trick when we compute $\underline{c}^{(\gamma,\beta)} = A\underline{c}^{(\alpha,\beta)}$: apply the first row directly, and use the algorithm described in Figure \ref{fig:Summary} on $\left(A_{jk}\right)_{1\leq j,k\leq N}$ for the remaining entries.


\subsubsection{Case 4: $\mathbf{|\alpha-\gamma|>1}$, but the difference is a noninteger}
If $|\alpha-\gamma|>1$, then either $\alpha<\gamma-1$ or $\alpha>\gamma+1$. If $\alpha<\gamma-1$, then we 
convert the Jacobi $(\alpha,\beta)$ expansion to $(\alpha+1,\beta)$ using Case 1, repeating if necessary until $|\alpha-\gamma|<1$. 
Similarly, if $\alpha>\gamma+1$, then we convert the Jacobi  $(\alpha,\beta)$ expansion to $(\alpha-1,\beta)$ using Case 1. Again, repeating 
until $|\alpha-\gamma|<1$. Thus, this case reduces the difference between $\alpha$ and $\gamma$ until the criterion for 
Case $2$ or $3$ is applicable.

\bigskip 

The four cases above are implemented in the {\tt jac2jac} commands in Chebfun~\cite{Chebfun} and FastTransforms.jl~\cite{FastTransforms} with the syntax 
{\tt jac2jac(v,a,b,g,d)}. Based on the particular values of {\tt a}, {\tt b}, {\tt g}, and {\tt d} various cases above are exercised. For all 
parameter ranges the cost of the conversion is at most $\mathcal{O}(N(\log N)^2)$ operations. This algorithm is also employed for 
the commands {\tt jac2cheb} and {\tt cheb2jac} by exploiting the fact that $T_k(x) = P_k^{(-1/2,-1/2)}(x)/P_k^{(-1/2,-1/2)}(1)$~\cite[(18.7.3)]{NISTHandbook}. One can also compute the Jacobi-to-Chebyshev and Chebyshev-to-Jacobi conversions in $\mathcal{O}(N(\log N)^2/\log\log N)$ 
operations using asymptotic expansions of Jacobi polynomials~\cite{slevinsky2016use}. 
 
\subsection{Laguerre-to-Laguerre conversion}\label{sec:AdditionalConversions}
We are not aware of major applications for Laguerre-to-Laguerre conversions, 
though related conversions are discussed in~\cite{Chriikjian_16_01}.  Due to the simplicity of the conversion, 
we include it in this section. 
 
 For $\alpha>-1$ the generalized Laguerre polynomial of degree $k$ is given by $L_k^{(\alpha)}(x) = \frac{\Gamma(n+\alpha+1)}{\Gamma(\alpha+1)\Gamma(n+1)}{}_1F_1(-n,\alpha+1;x)$~\cite[Tab.~18.5.12]{NISTHandbook}. The sequence $L_0^{(\alpha)}, L_1^{(\alpha)},\ldots$ forms a family of polynomials that are orthogonal with respect to the weight function $x^\alpha e^{-x}$ on $[0,\infty)$. 
 
 Suppose that $\alpha_1>-1$ and $\alpha_2>-1$. Then, there is an 
 upper-triangular matrix, $A$, that converts expansion coefficients in the $L^{(\alpha_1)}$ basis 
 to coefficients in the $L^{(\alpha_2)}$ basis.  If $\alpha_1-\alpha_2$ is an integer, then an $\mathcal{O}(N)$ complexity
 algorithm for computing the matrix-vector product is almost immediate thanks to the recurrence relation given in~\cite[(18.9.13)]{NISTHandbook}. 
 If $\alpha_1-\alpha_2$ is not an integer, then there is an explicit formula for the entries of $A$ given by~\cite[(18.18.18)]{NISTHandbook}
 \[
  A_{jk} = \begin{cases} \frac{1}{\Gamma(\alpha_1 - \alpha_2)} \frac{\Gamma(k-j+\alpha_1-\alpha_2)}{\Gamma(k-j+1)}, & 0\leq j\leq k\leq N,\\ 0, & \text{otherwise}.\end{cases}
 \]
 One observes that this conversion matrix is a diagonally-scaled Toeplitz matrix, which is also a diagonally-scaled Toeplitz-dot-Hankel matrix by taking the 
 Hankel part as the matrix of all ones. A fast $\mathcal{O}(N \log N)$ algorithm follows by a fast Toeplitz matrix-vector product based on the FFT~\cite{Ng_04_01}.

\section*{Conclusion} 
Many of the standard conversion matrices for converting between expansions coefficients 
in orthogonal polynomial basis can be written as a diagonally-scaled Hadamard product 
between a Toeplitz and Hankel matrix. This leads to an $\mathcal{O}(N(\log N)^2)$ complexity 
for basis conversion for a polynomial of degree $N$. The resulting algorithm is conceptually
simple, FFT-based, and requires no precompution, while being competitive in terms of 
computational time with existing fast algorithms. 

\section*{Acknowledgments}
We thank the School of Mathematics and Statistics at The University of Sydney for awarding Alex Townsend a travel grant that 
allowed him to visit the school in February 2016. We thank the Cecil King Foundation and the London Mathematical Society 
for awarding Marcus Webb the Cecil King Travel Scholarship to visit The University of Sydney from January 
to April 2016.  We thank Laurent Demanet and Haihao Lu for serendipitously discussing Toeplitz-dot-Hankel 
matrices with the first author a year earlier. We are grateful to Bernhard Beckermann for a discussion on the 
singular values of real, symmetric, and positive semidefinite Hankel matrices. We also thank Nick Hale and 
Mika\"el Slevinsky for discussions on related topics over many years and providing excellent comments that 
improved this paper and our implementations of the algorithms.


\begin{thebibliography}{3}
\bibitem{Alpert_91_01} {\sc B. K. Alpert and V. Rokhlin}, {\em A fast algorithm for the evaluation of Legendre expansions}, {SIAM J. Sci. Stat. Comput.}, 12 (1991), pp.~158--179.
\bibitem{Askey_75_01} {\sc R. Askey}, {\em Orthogonal Polynomials and Special Functions}, {SIAM}, 1975.
\bibitem{Beckermann_00_01} {\sc B. Beckermann}, {\em The condition number of real Vandermonde, Krylov and positive definite Hankel matrices}, {Numer. Math.}, 85 (2000), pp.~553--577.
\bibitem{Bezanson_14_01} {\sc J. Bezanson, A. Edelman, S. Karpinski, and V. B. Shah}, {\em Julia: A fresh approach to numerical computing}, {arXiv preprint arXiv:1411.1607}, (2014).
\bibitem{Cantero_12_01} {\sc M. J. Cantero and A. Iserles}, {\em On rapid computation of expansions in ultraspherical polynomials}, {SIAM J.\ Numer.\ Anal.}, 50 (2012), pp.~307--327.
\bibitem{Chriikjian_16_01} {\sc G. S. Chirikjian and A. B. Kyatkin}, {\em Harmonic Analysis for Engineers and Applied Scientists}, {CRC Press}, {Second edition}, 2000
\bibitem{Don_94_01} {\sc W.\ S.\ Don and D. Gottlieb}, {\em The Chebyshev--Legendre method: Implementing Legendre methods on Chebyshev points}, {SIAM J.\ Numer.\ Anal.}, 31 (1994), pp.~1519--1534.  
\bibitem{Chebfun} {\sc T. A. Driscoll, N. Hale, and L. N. Trefethen}, {\em Chebfun Guide}, {Pafnuty Publications}, 2014.
\bibitem{Frigo_05_01} {\sc M. Frigo and S. G. Johnson}, {\em The design and implementation of FFTW3}, {Proc. IEEE}, 93 (2005), pp.~216--231.
\bibitem{Golub_13_01} {\sc G. H. Golub and C. F. Van Loan}, {\em Matrix Computations}, {4th edition}, {Johns Hopkins University Press}, 2013. 
\bibitem{Gentleman_72_01} {\sc M. W. Gentleman}, {\em Implementing Clenshaw-Curtis quadrature, II computing the cosine transformation}, {Comm. ACM}, 15 (1972), pp.~343--346.
\bibitem{Grasedyck_01_01} {\sc L. Grasedyck}, {\em Singular value bounds for the Cauchy matrix and solutions of Sylvester equations},
{Technical report}, {University of Kiel}, 13, 2001.
\bibitem{Gu_04_01} {\sc M. Gu and L. Miranian}, {\em Strong rank revealing Cholesky factorization}, {Electr. Trans. Numer. Anal.}, 17 (2004), pp.~76--92.
\bibitem{Hale_14_01} {\sc N.\ Hale and A. Townsend}, {\em A fast, simple, and stable Chebyshev--Legendre transform using an asymptotic formula}, {SIAM J.\ Sci.\ Comput.}, 36 (2014), A148--A167.
\bibitem{Hale_14_02} {\sc N.\ Hale and A. Townsend}, {\em An algorithm for the convolution of Legendre series}, {SIAM J. Sci. Comput.}, 36 (2014), A1207--A1220.
\bibitem{Hale_15_01} {\sc N.\ Hale and A. Townsend}, {\em A fast FFT-based discrete Legendre transform}, to appear in IMA Numer. Anal., 2016.
\bibitem{Harbrecht_12_01} {\sc H. Harbrecht, M. Peters, and R. Schneider}, {\em On the low-rank approximation by the pivoted Cholesky decomposition}, {Appl. Numer. Math.}, 62 (2012), pp.~428--440.
\bibitem{Higham_90_01} {\sc N. J. Higham}, {\em Analysis of the Cholesky decomposition of a semi-definite matrix}, {in Reliable Numerical Computation}, 1990.
\bibitem{Iserles_11_01} {\sc A.\ Iserles}, {\em A fast and simple algorithm for the computation of Legendre coefficients}, {Numer.\ Math.}, 117 (2011), pp.~529--553.
\bibitem{Keiner_09_01} {\sc J. Keiner}, {\em Computing with expansions in Gegenbauer polynomials}, {SIAM J. Sci. Comput.}, 31 (2009), pp.~2151--2171.
\bibitem{Keiner_11_01} {\sc J. Keiner}, {\em Fast Polynomial Transforms}, {Logos Verlag Berlin GmbH}, 2011.
\bibitem{Maroni_08_01} {\sc P. Maroni and Z. da Rocha}, {\em Connection coefficients between orthogonal polynomials and the canonical sequence: an approach based on symbolic computation}, {Numer. Algor.}, 47 (2008), pp.~291--314.
\bibitem{Mason_02_01} {\sc J.\ C.\ Mason and D.\ C.\ Handscomb}, {\em Chebyshev Polynomials}, {Taylor \& Francis}, (2002).
\bibitem{Matsaglia_74_01} {\sc G. Matsaglia and G. P. H. Styan}, {\em Equalities and inequalities for ranks of matrices}, {Lin.
Multi. Alg.}, 2 (1974), pp.~269--292.
\bibitem{Mori_99_01} {\sc A.\ Mori, R.\ Suda, and M.\ Sugihara}, {\em An improvement on Orszag's fast algorithm for Legendre polynomial transform}, {Trans.\ Info.\ Processing Soc.\ Japan}, 40 (1999), pp.~3612--3615.
\bibitem{Ng_04_01} {\sc M. K. Ng}, {\em Iterative Methods For Toeplitz Systems}, {Oxford University Press}, 2004. 
\bibitem{ApproxFun} {\sc S. Olver}, {\sc R. M. Slevinsky}, et al., {{\tt https://github.com/ApproxFun/ApproxFun.jl}}, v0.1.0, 2016. 
\bibitem{NISTHandbook} {\sc F. W. J. Olver, D. W. Lozier, R. F. Boisvert, and C. W. Clark},  {\em NIST Handbook of Mathematical Functions},  Cambridge University Press, 2010.
\bibitem{Olver_13_01} {\sc S. Olver and A. Townsend}, {\em A fast and well-conditioned spectral method}, {SIAM Review}, 55 (2013), pp.~462--489.
\bibitem{Orszag_86_01} {\sc S. A. Orszag}, {\em Fast eigenfunction transforms}, {Science and Computers}, {Academic Press, New York}, (1986), pp.~23--30.
\bibitem{Peller_12_01} {\sc V. Peller}, {\em Hankel Operators and Their Applications}, {Springer}, 2012. 
\bibitem{Potts_98_01} {\sc D.\ Potts, G.\ Steidl, and M.\ Tasche}, {\em Fast algorithms for discrete polynomial transforms}, {Math.\ Comp.}, 67 (1998), pp.~1577--1590.
\bibitem{Shen_94_01} {\sc J.\ Shen}, {\em Efficient spectral-Galerkin method I. Direct solvers of second- and fourth-order equations using Legendre polynomials}, {SIAM J.\ Sci.\ Comput.}, 15 (1994), pp.~1489--1505.
\bibitem{Shen_16_01} {\sc J. Shen, Y. Wang, and J. Xia}, {\em Fast structured Jacobi-Jacobi transforms}, preprint, 2016.  
\bibitem{slevinsky2016use} {\sc R. M. Slevinsky}, {\em On the use of Hahn's asymptotic formula and stabilized recurrence for a fast, simple, and stable Chebyshev--Jacobi transform}, arXiv preprint arXiv:1602.02618, 2016.
\bibitem{FastTransforms} {\sc R. M. Slevinsky}, {\sc S. Olver}, et al.  {{\tt https://github.com/ApproxFun/ApproxFun.jl}}, v0.0.4, 2016. 
\bibitem{Townsend_15_01} {\sc A. Townsend and L. N. Trefethen}, {\em Continuous analogues of matrix factorizations}, {Proc. R. Soc. A.}, 471 (2015). 
\bibitem{Trefethen_13_01} {\sc L. N. Trefethen}, {\em Approximation Theory and Approximation Practice}, {SIAM}, 2013. 
\bibitem{Tricomi_57_01} {\sc F. G. Tricomi and A. Erd\'{e}lyi}, {\em The asymptotic expansion of a ratio of gamma functions}, {Pacific J. Math.}, 1 (1951), pp.~133--142.
\bibitem{Wang_14_01} {\sc H. Wang and D. Huybrechs}, {\em Fast and accurate computation of Jacobi expansion coefficients of analytic functions}, submitted, (2014).
\bibitem{Xu_08_01} {\sc W. Xu and S. Qiao}, {\em A fast symmetric SVD algorithm for square Hankel matrices}, {Lin. Alg. Appl.}, 428 (2008), pp.~550--563.
\end{thebibliography}
\end{document}